\documentclass[12pt]{amsart}

\usepackage{amssymb,amsthm,amsmath,url,relsize,etoolbox,cite,nameref,enumerate,fancyhdr}
\usepackage{mathtools,xcolor,enumitem,bm,bbm}
\usepackage[utf8]{inputenc}
\usepackage[hidelinks]{hyperref}
\usepackage{mathrsfs}
\usepackage{mnsymbol}
\usepackage{tabularx}
\usepackage{nicematrix}
\theoremstyle{plain}
\newtheorem*{theorem*}{Theorem}
\newtheorem{theorem}[subsubsection]{Theorem}

\newtheorem{lemma}[subsubsection]{Lemma}

\newtheorem{definition}[subsubsection]{Definition}

\DeclareMathOperator{\degree}{deg}
\DeclareMathOperator{\modulus}{mod}

\DeclareMathOperator{\rank}{rank}

\DeclareMathOperator{\reduced}{Red}
\DeclareMathOperator{\partition}{Part}

\DeclareMathOperator{\hank}{Hank}

\newcommand*\mathbfe[1]{\boldsymbol{#1}}
\newcommand{\rhopi}{$(\rho , \pi )$}
\newcommand*\mean[2]{\textrm{Mean}_{#1} ({#2})}
\newcommand*\variance[2]{\textrm{Var}_{#1} ({#2})}

\newcommand{\multisetmatrix}{\operatorname{MS}}

\makeatletter
\def\l@subsection{\@tocline{2}{0pt}{2.5pc}{5pc}{}}
\makeatother

\makeatletter
\renewcommand*\env@matrix[1][*\c@MaxMatrixCols c]{%
  \hskip -\arraycolsep
  \let\@ifnextchar\new@ifnextchar
  \array{#1}}
\makeatother

\begin{document}

\title{The Variance of the Sum of Two Squares over Intervals in $\mathbb{F}_q [T]$: I}
\author{Michael Yiasemides}
\date{\today}
\address{Department of Mathematics, University of Exeter, Exeter, EX4 4QF, UK}
\email{my298@exeter.ac.uk}
%\email{michael.yiasemides@gmail.com}

\maketitle

\begin{abstract}
For $B \in \mathbb{F}_q [T]$ of degree $2n \geq 2$, consider the number of ways of writing\break $B=E^2 + \gamma F^2$, where $\gamma \in \mathbb{F}_q^*$ is fixed, and $E,F \in \mathbb{F}_q [T]$ with $\degree E = n$ and $\degree F = m < n$. We denote this by $S_{\gamma ; m} (B)$. We obtain an exact formula for the variance of $S_{\gamma ; m} (B)$ over intervals in $\mathbb{F}_q [T]$. We use the method of additive characters and Hankel matrices that the author previously used for the variance and correlations of the divisor function. In Section \ref{section, motivation}, we give a short overview of our approach; and we briefly discuss the possible extension of our result to the number of ways of writing $B=E^2 + T F^2$.
\end{abstract}

%\maketitle
%\thispagestyle{fancy}
%\pagenumbering{gobble}

\allowdisplaybreaks

%\tableofcontents

\section{Introduction} \label{section, introduction}

The representation of natural numbers as a sum of powers is a topic in number theory that has a history stretching back to ancient times. During the 3rd century AD, Diophantus of Alexandria theorised that every positive integer can be expressed as the sum of four squares. This was proved by Lagrange in 1770. Waring generalised this by conjecturing that for each positive integer $k$ there exists a positive integer $g(k)$ such that every natural number can be expressed as the sum of at most $g(k)$ $k$-th powers of natural numbers. This was proved by various authors. \\

Further extensions of this problem exist, and for a more detailed account we refer the reader to the survey by Vaughan and Wooley \cite{VaughanWooley2002_WaringsProblemSurvey}. A related problem is the natural numbers that can be expressed as a sum of two squares, and how they are distributed in intervals. Specifically, for integers $n \geq 0$ let
\begin{align*}
b(n) 
= \begin{cases}
1 &\text{ if $n=a^2 +b^2$ for some integers $a,b \geq 0$;}\\
0 &\text{ otherwise.}
\end{cases}
\end{align*}
We are interested in the mean and variance of this function in intervals. Landau \cite{Landau1908_UberEinteilungPositivenGanzenZahlenKlassen} proved
\begin{align*}
\frac{1}{x} \sum_{n \leq x} b(n)
\sim \bigg( \frac{1}{\sqrt{2}} \prod_{p \equiv 3 (\modulus 4)} (1-p^{-2} )^{-\frac{1}{2}} \bigg) \frac{1}{\sqrt{ \log x}}
\end{align*}
as $x \rightarrow \infty$. Regarding more general intervals, Hooley \cite{Hooley1994_IntervBetweenNumbSumTwoSquareIV} proved unconditionally for $\frac{7}{12} < \epsilon < 1$ that 
\begin{align*}
\frac{1}{2 x^{\epsilon} } \sum_{\lvert n-x \rvert \leq x^{\epsilon}} b(n)
\sim \bigg( \frac{1}{\sqrt{2}} \prod_{p \equiv 3 (\modulus 4)} (1-p^{-2} )^{-\frac{1}{2}} \bigg) \frac{1}{\sqrt{ \log x}}
\end{align*}
as $x \rightarrow \infty$. Conditional on the Riemann hypothesis for $\zeta (s)$ and for the Dirichlet $L$-function associated to the non-trivial character of modulus $4$, Hooley proved the above for $\frac{1}{2} < \epsilon < 1$. It is conjectured that the result actually holds for $0 < \epsilon < 1$. \\

We can study this problem in function fields, for which we need to establish some notation. Let $\mathcal{A} := \mathbb{F}_q [T]$, the polynomial ring over the finite field of order $q$ (a prime power). For $A \in \mathcal{A} \backslash \{ 0 \}$, we define $\lvert A \rvert := q^{\degree A}$, and $\lvert 0 \rvert := 0$. We denote the set of monics by $\mathcal{M}$. For $\mathcal{S} \subseteq \mathcal{A}$ and integers $n \geq 0$, we define $\mathcal{S}_n , \mathcal{S}_{\leq n} , \mathcal{S}_{<n}$ to be the set of elements in $\mathcal{S}$ with degree $n , \leq n , < n$, respectively (we take $\degree 0$ to be $-\infty$). For $A \in \mathcal{A}$ and $h \geq 0$, we define the interval of centre $A$ and radius $h$ by $I(A;h) := \{ B \in \mathcal{A} : \degree (B-A) < h \}$. Finally, for $B \in \mathcal{M}$, we define
\begin{align} \label{statement, b_q,T definition}
b_{q;T} (B)
:= \begin{cases}
1 &\text{ if $B = E^2 + T F^2$ for some $E,F \in \mathcal{M}$;} \\
0 &\text{ otherwise.}
\end{cases}
\end{align}
We attempt to remain as consistent as possible with the notation established in \cite{Bary-SorokerSmilanskyWolf2016_FunFieldAnalogueLandauTheoSumSquare}. \\

Bary-Soroker, Smilansky, and Wolf \cite{Bary-SorokerSmilanskyWolf2016_FunFieldAnalogueLandauTheoSumSquare} prove that
\begin{align*}
\frac{1}{q^n } \sum_{B \in \mathcal{M}_n } b_{q;T} (B)
= \frac{1}{4^n } \binom{2n}{n} + O_n (q^{-1}) 
\end{align*}
as $q^n \rightarrow \infty$. There are several ways that $q^n$ can tend to infinity. The above is valid when $q$ is much larger than $n$. For all ways that $q^n$ can tend to infinity, Gorodetsky \cite{Gorodetsky2017_PolyAnalogueLandauTheoRelateProb} proved that
\begin{align*}
\frac{1}{q^n } \sum_{B \in \mathcal{M}_n } b_{q;T} (B)
= \bigg( (1-q^{-1})^{-\frac{1}{2}} \prod_{\chi_2 (P) = -1} \big( 1 - q^{-2 \degree P})^{-\frac{1}{2}} \bigg)
	\binom{n-\frac{1}{2}}{n} \bigg( 1 + O_n \Big( \frac{1}{qn} \Big) \bigg) ,
\end{align*}
where $\chi_2$ is the non-trivial quadratic Dirichlet character of modulus $T$. \\

For more general intervals in function fields, Bank, Bary-Soroker, and Fehm \cite{BankBary-SorokerFehm2018_SumTwoSquareShortIntervPolyRingFinField} prove the following for the mean value. For odd $q$ odd, $n>2$, $3 \leq h \leq n$, and $A \in \mathcal{M}_n$, we have
\begin{align*}
\frac{1}{q^{h} } \sum_{B \in I(A;h)} b_{q;T} (B)
= \frac{1}{4^n } \binom{2n}{n} + O_n (q^{-\frac{1}{2}}) 
\end{align*}
as $q \rightarrow \infty$. Gorodetsky and Rodgers \cite{GorodetskyRodgers2021_VarNumbTwoSquareFqTShortInterv} obtain an asymptotic formula as $q \longrightarrow \infty$ for the variance. \\

In this paper, we are interested in functions that are slightly different to (\ref{statement, b_q,T definition}). First, for $B \in \mathcal{M}$ with $\degree B \geq 1$, we define
\begin{align*}
S_{T} (B)
:= \lvert \{ (E,F) \in \mathcal{M}^2 : B = E^2 + T F^2 \} \rvert .
\end{align*}
Note that if $\degree B = 2n$, then we must have $\degree E = n$ and $\degree F < n-1$; while if $\degree B = 2n+1$, then we must have $\degree E \leq n$ and $\degree F = n$. The function $S_{T} (B)$ differs from (\ref{statement, b_q,T definition}) only in that it counts the number of ways of expressing $B$ as a sum of squares, instead of simply being an indicator function. We briefly discuss the function $S_{T} (B)$ further in Section \ref{section, motivation}. \\

The second function we are interested in is $S_{\gamma} (B)$ which is similar to $S_{T} (B)$ but we have a fixed $\gamma \in \mathbb{F}_q^*$ in the place of $T$. That is, for $B \in \mathcal{M}$ of degree $2n \geq 2$, we define
\begin{align*}
S_{\gamma} (B)
:= \lvert \{ (E,F) \in \mathcal{M}_{n} \times \mathcal{M}_{<n} : B = E^2 + \gamma F^2 \} \rvert .
\end{align*}
Note that when $\gamma \neq 1$, we are forced to have $\degree E = n$ and $\degree F < n$, and so we could simply write $(E,F) \in \mathcal{M}^2$ in the definition, instead of $(E,F) \in \mathcal{M}_{n} \times \mathcal{M}_{<n}$. However, when $\gamma = 1$, this is not possible as the roles of $E$ and $F$ would be symmetric, and it would make $S_{1} (B)$ twice the value we would like it to be. Again, we briefly discuss this function further in Section \ref{section, motivation}, and it will be the focus of Part II of this paper. \\

The third function, and the one we established results for in this paper, is defined as follows. For $0 \leq m \leq n-1$,
\begin{align*}
S_{\gamma ; m} (B)
:= \lvert \{ (E,F) \in \mathcal{M}_{n} \times \mathcal{M}_{m} : B = E^2 + \gamma F^2 \} \rvert .
\end{align*}
We are interested in the mean and variance of $S_{\gamma ; m} (B)$ over intervals. Thus, for $0 \leq h \leq 2n$, we define
\begin{align*}
\mean{2n,h}{S_{\gamma ; m}}
:= \frac{1}{q^{2n}} \sum_{A \in \mathcal{M}_{2n}} \sum_{B \in I(A;h)} S_{\gamma ; m} (B) ,
\end{align*}
and we have that
\begin{align*}
\mean{2n,h}{S_{\gamma ; m}}
= \frac{q^h}{q^{2n}} \sum_{B \in \mathcal{M}_{2n}} S_{\gamma ; m} (B) 
= \frac{q^h}{q^{2n}} \sum_{\substack{ E \in \mathcal{M}_{n} \\ F \in \mathcal{M}_m}} 1 
=  q^{m+h-n} ;
\end{align*}
and we define the variance
\begin{align}
\begin{split} \label{statement, intro section, S_gamma,m variance}
\variance{2n,h}{S_{\gamma ; m}}
:= &\frac{1}{q^{2n}} \sum_{A \in \mathcal{M}_2n} 
	\Big( \sum_{B \in I(A;h)} S_{\gamma ; m} (B) - \mean{2n,h}{S_{\gamma ; m}} \Big)^2 \\
= &\frac{1}{q^{2n}} \sum_{A \in \mathcal{M}_{2n}} \Big( \sum_{B \in I(A;h)} S_{\gamma ; m} (B) \Big)^2 
	- q^{2(m+h-n)} . \\
\end{split}
\end{align}

We prove the following theorem.

\begin{theorem} \label{main theorem, variance sums of squares}
For $m \geq 1$ and $m+1 \leq n \leq 2m-1$, we have
\begin{align*}
\variance{2n,h}{S_{\gamma ; m}}
= \begin{cases}
0 &\text{ if $m \leq h \leq 2n$;} \\
\frac{q^h}{q^n} (q^m + q^h ) &\text{ if $2m-n \leq h \leq m-1$;} \\
\frac{q^{m+h}}{q^{2n}} \bigg[ q^n + q^m + (2m-n-h) (q-1) q^{m-1} \bigg] &\text{ if $0 \leq h \leq 2m-n-1$.}
\end{cases} \\
\end{align*}

For $m \geq 1$ and $n=2m$, we have
\begin{align*}
\variance{2n,h}{S_{\gamma ; m}}
= \begin{cases}
0 &\text{ if $m \leq h \leq 2n$;} \\
\frac{q^h}{q^n} (q^m + q^h ) &\text{ if $0 \leq h \leq m-1$.} 
\end{cases} \\
\end{align*}

For $m \geq 1$ and $n \geq 2m+1$, we have
\begin{align*}
\variance{2n,h}{S_{\gamma ; m}}
= \begin{cases}
0 &\text{ if $n \leq h \leq 2n$;} \\
\frac{q^{2m+h}}{q^{2n}} (q^n - q^h ) &\text{ if $2m \leq h \leq n-1$;} \\
\frac{q^{2h}}{q^{2n}} (q^n - q^{2m}) &\text{ if $m \leq h \leq 2m-1$;} \\
\frac{q^{m+h}}{q^{2n}} (q^n - q^{m+h}) &\text{ if $0 \leq h \leq m-1$.} 
\end{cases} \\
\end{align*}

For $m=0$ and $n \geq 1$, we have
\begin{align*}
\variance{2n,h}{S_{\gamma ; m}}
= \begin{cases}
0 &\text{ if $n \leq h \leq 2n$;} \\
\frac{q^h}{q^{2n}} (q^n - q^h ) &\text{ if $0 \leq h \leq n-1$.} 
\end{cases} 
\end{align*}
\end{theorem}

In this theorem, we have several cases and subcases because we allow $n$, $m$, and $h$ to be fixed. In Part II of this paper, where we consider the variance of $S_{\gamma} (B)$, the result is more succinct because $m$ can take any permissible value and is not fixed. For the same reason though, it does require us to alter our approach to an extent. We discuss the extension of our results to $S_{\gamma} (B)$ and $S_{T} (B)$ at the end of Section \ref{section, motivation}. At the beginning of that section we discuss how we prove Theorem \ref{main theorem, variance sums of squares} using the method of additive characters and Hankel matrices that the author used in \cite{Yiasemides2021_VariCorrDivFuncFpTHankelMatr_ArXiv_v2} for the variance and correlations of the divisor function. In Section \ref{section, Hankel matrices} we develop some theory on Hankel matrices, and in Section \ref{section, main theorems proofs} we prove Theorem \ref{main theorem, variance sums of squares}.

\section{Motivation} \label{section, motivation}

Suppose we have two polynomials $A,B \in \mathcal{A}$ such that
\begin{align*}
A = &a_0 + a_1 T + \ldots + a_m T^m , \\
B = &b_0 + b_1 T + \ldots + b_n T^n .
\end{align*}
We note that the product $AB$ is equal to
\begin{align} \label{statement, motivation section, AB as matrix mult}
\begin{matrix} 
\begin{pNiceMatrix}
a_0 & a_1 & \Cdots &  &  &  & a_m
\end{pNiceMatrix}
\\  \\  \\  \\  \\  \\
\end{matrix}
\begin{pNiceMatrix}
1  & T & T^2 & \Cdots & \Cdots             & \Cdots              & T^n \\
T  & T^2 &              &           &                        &                        & \Vdots \\
T^2  &              &              &           &                        &                        & \Vdots \\
\Vdots     &              &              &           &                        &                        & \Vdots \\
\Vdots     &              &              &           &                        &                        & T^{m+n-2} \\
\Vdots     &              &              &           &                        & T^{m+n-2} & T^{m+n-1} \\
T^m & \Cdots    & \Cdots    & \Cdots & T^{m+n-2} & T^{m+n-1} & T^{m+n} 
\end{pNiceMatrix}
\begin{pNiceMatrix} b_0 \\ b_1 \\ \Vdots \\  \\  \\  \\ b_n \end{pNiceMatrix} .
\end{align}
Indeed, the $i$-th coefficient of $AB$ is 
\begin{align*}
\sum_{\substack{0 \leq m_1 \leq m \\ 0 \leq n_1 \leq n \\ m_1 + n_1=i}} a_{m_1} b_{n_1} ,
\end{align*}
which is the same as the $i$-th coefficient of (\ref{statement, motivation section, AB as matrix mult}) above. Now, the matrix in (\ref{statement, motivation section, AB as matrix mult}) is such that all entries on a given skew-diagonal are identical, and thus it is a Hankel matrix. This is similar to a Toeplitz matrix, where the entries on a given diagonal are identical. From (\ref{statement, motivation section, AB as matrix mult}), we can immediately see how Hankel matrices are related to polynomial multiplication. \\

In the remainder of this paper, we will only work with Hankel matrices over $\mathbb{F}_q$. Formally, an $l \times m$ Hankel matrix over $\mathbb{F}_q$ is a matrix of the form
\begin{align*}
(\beta_{i+j-2} )_{\substack{1 \leq i \leq l \\ 1 \leq j \leq m }}
= \begin{pNiceMatrix}
\beta_0  & \beta_1 & \beta_2 & \Cdots & \Cdots             & \Cdots              & \beta_{l-1} \\
\beta_1  & \beta_2 &              &           &                        &                        & \Vdots \\
\beta_2  &              &              &           &                        &                        & \Vdots \\
\Vdots     &              &              &           &                        &                        & \Vdots \\
\Vdots     &              &              &           &                        &                        & \beta_{l+m-4} \\
\Vdots     &              &              &           &                        & \beta_{l+m-4} & \beta_{l+m-3} \\
\beta_m & \Cdots    & \Cdots    & \Cdots & \beta_{l+m-4} & \beta_{l+m-3} & \beta_{l+m-2} \\
\end{pNiceMatrix} ,
\end{align*}
where $\beta_0 , \ldots , \beta_{l+m-2} \in \mathbb{F}_q$. We denote the set of all $l \times m$ Hankel matrices over $\mathbb{F}_q$ by $\mathscr{H}_{l,m}$. It will be required later to define $\mathscr{H}_{l,m}^h$ to be the set of $l \times m$ Hankel matrices over $\mathbb{F}_q$ with first $h$ skew-diagonals equal to zero (that is, in terms of the matrix above, we have $\beta_0 , \beta_1 , \ldots , \beta_{h-1} = 0$). For square matrices, let us also define $\mathscr{H}_{n,n}^h (r)$ to be the set of $n \times n$ Hankel matrices over $\mathbb{F}_q$ with first $h$ skew-diagonals equal to zero and with rank equal to $r$. \\

We can see that for a sequence $\mathbfe{\alpha} = (\alpha_0 , \alpha_1 , \ldots , \alpha_n ) \in \mathbb{F}_q^{n+1}$, there are $n+1$ Hankel matrices associated to $\mathbfe{\alpha}$, which depend on the number of rows and columns. These are the matrices $(\alpha_{i+j-2} )_{\substack{1 \leq i \leq l+1 \\ 1 \leq j \leq m+1 }}$ where $l,m \geq 0$ and $l+m=n$, which we denote by $H_{l+1 , m+1} (\mathbfe{\alpha})$. \\

Let us make a brief notational remark. Let $M$ be an $l \times m$ matrix and let $l_1 , l_2 , m_1 , m_2$ satisfy $l_1 + l_2 \leq l$ and $m_1 + m_2 \leq m$. Then, we define $M[l_1 , -l_2 ; m_1 , -m_2]$ to be the submatrix of $M$ consisting of the first $l_1$ and last $l_2$ rows, and the first $m_1$ and last $m_2$ columns. In the special cases when one or more of $l_1 , l_2 , m_1 , m_2$ are zero, we may not include them. For example, $M[l_1 ; -m_2]$ should be taken to be $M[l_1 , 0 ; 0 , -m_2 ]$. For a Hankel matrix $H_{l+1 , m+1} (\mathbfe{\alpha})$ as above, and $l_1 \leq l$ and $m_1 \leq m$, we may simply write $H_{l_1 , m_1} (\mathbfe{\alpha})$ in the place of $H_{l+1 , m+1} (\mathbfe{\alpha}) [l_1 ; m_1]$. \\

Regarding sums of squares, it is obvious that squaring a polynomial is a kind of multiplication, and so it is not surprising that Hankel matrices will appear. Specifically, we work with the equation $B = E^2 + \gamma F^2$, and Hankel matrices appear when we compare the coefficients on both sides, similar to what we did in (\ref{statement, motivation section, AB as matrix mult}). \\

To compare coefficients, we will use a non-trivial additive character and its orthogonality relation. An additive character $\psi$ on $\mathbb{F}_q$ is a function from $\mathbb{F}_q$ to $\mathbb{C}^*$ satisfying $\psi (a+b) = \psi (a) \psi (b)$ (note this implies $\psi (0) =1$ and $\psi (-a) = \psi (a)^{-1}$ for all $a \in \mathbb{F}_q$). We say $\psi$ is non-trivial if there exists $a \in \mathbb{F}_q^*$ such that $\psi (a) \neq 1$, and in this case we have the orthogonality relation
\begin{align} \label{statement, intro, exp identity sum}
\frac{1}{q} \sum_{\alpha \in \mathbb{F}_q } \psi (\alpha b)
= \begin{cases}
1 &\text{if $b=0$,} \\
0 &\text{if $b \in \mathbb{F}_q^*$.}
\end{cases}
\end{align}
The first case follows from the fact that if $b=0$, then $\psi (\alpha b) = 1$ for all $\alpha \in \mathbb{F}_q$. The second case follows from the fact that if $b \in \mathbb{F}_q^*$, then $\alpha b$ and $\alpha b + a$ both vary over $\mathbb{F}_q$ as $\alpha$ varies over $\mathbb{F}_q$, and so
\begin{align*}
\sum_{\alpha \in \mathbb{F}_q } \psi (\alpha b)
= \sum_{\alpha \in \mathbb{F}_q } \psi (\alpha b + a)
= \psi (a) \sum_{\alpha \in \mathbb{F}_q } \psi (\alpha b).
\end{align*}
Since $\psi (a) \neq 1$, we deduce that $\sum_{\alpha \in \mathbb{F}_q } \psi (\alpha b) =0$. In the remainder of this article, $\psi$ is a non-trivial character on $\mathbb{F}_q$. \\

We are now in a position to consider the variance of $S_{\gamma ; m} (B)$. From (\ref{statement, intro section, S_gamma,m variance}), it suffices to consider
\begin{align} \label{statement, motivation section, variance S_gamma,m minus the mean squared}
\sum_{A \in \mathcal{M}_{2n}} \Big( \sum_{B \in I(A;h)} S_{\gamma ; m} (B) \Big)^2 .
\end{align}
For $B \in \mathcal{M}_{2n}$, we have
\begin{align*}
S_{\gamma ; m} (B) 
= \sum_{\substack{E \in \mathcal{M}_n \\ F \in \mathcal{M}_m \\ E^2 + \gamma F^2 = B }} 1 
= \sum_{\substack{E \in \mathcal{M}_n \\ F \in \mathcal{M}_m }} \mathbbm{1} (E^2 + \gamma F^2 = B) 
= \sum_{\substack{E \in \mathcal{M}_n \\ F \in \mathcal{M}_m }} 
	\prod_{i=0}^{2n} \mathbbm{1} (\{ E^2 \}_i + \gamma \{ F^2 \}_i = \{ B \}_i ) ,
\end{align*}
where for a proposition $\mathrm{P}$, we define $\mathbbm{1} (\mathrm{P})$ to be $1$ if $\mathrm{P}$ is true and $0$ if $\mathrm{P}$ is false, and for $C \in \mathcal{A}$ we define $\{ C \}_i$ to be its $i$-th coefficient. Now, by (\ref{statement, intro, exp identity sum}), we have
\begin{align*}
\mathbbm{1} (\{ E^2 \}_i + \gamma \{ F^2 \}_i = \{ B \}_i )
= \sum_{\alpha_i \in \mathbb{F}_q } \psi \Big( \alpha_i (\{ E^2 \}_i + \gamma \{ F^2 \}_i - \{ B \}_i ) \Big) .
\end{align*}
Hence,
\begin{align*}
S_{\gamma ; m} (B) 
= &\sum_{\substack{E \in \mathcal{M}_n \\ F \in \mathcal{M}_m }} 
	\prod_{i=0}^{2n} \sum_{\alpha_i \in \mathbb{F}_q } 
		\psi \Big( \alpha_i (\{ E^2 \}_i + \gamma \{ F^2 \}_i - \{ B \}_i ) \Big) \\
= &\sum_{\substack{E \in \mathcal{M}_n \\ F \in \mathcal{M}_m }} 
	\sum_{\mathbfe{\alpha} \in \mathbb{F}_q^{2n+1} }
	\prod_{i=0}^{2n} \psi \Big( \alpha_i (\{ E^2 \}_i + \gamma \{ F^2 \}_i - \{ B \}_i ) \Big) \\
= &\sum_{\mathbfe{\alpha} \in \mathbb{F}_q^{2n+1} }
	\sum_{\substack{\mathbf{e} \in \mathbb{F}_q^{n} \times \{ 1 \} \\ \mathbf{f} \in \mathbb{F}_q^{m} \times \{ 1 \}  }} 	
	\psi \Big( \mathbfe{e}^T H_{n+1 , n+1} (\mathbfe{\alpha}) \mathbfe{e} 
		+ \mathbfe{f}^T \gamma H_{m+1 , m+1} (\mathbfe{\alpha}) \mathbfe{f}) \Big)
	\psi (\mathbfe{\alpha} \cdot \mathbf{b}) ,
\end{align*}
where $\mathbfe{\alpha} = (\alpha_0 , \ldots , \alpha_{2n})$, and $\mathbf{b}, \mathbf{e} , \mathbf{f}$ are the vectors whose $i$-th entry is equal to the\break  $i$-th coefficient of $B,E,F$, respectively. To see how the Hankel matrices appear, one must perform a few simple rearrangements, but it is not difficult and it is similar to (\ref{statement, motivation section, AB as matrix mult}). Now, we substitute the above into (\ref{statement, motivation section, variance S_gamma,m minus the mean squared}). We do not get into the details of the calculations in this section, but in short we can remove the sums over $A,B$ and this simply creates a restriction on the Hankel matrices $H_{n+1 , n+1} (\mathbfe{\alpha})$ and $H_{m+1 , m+1} (\mathbfe{\alpha})$ that we work with, or, more specifically, the underlying sequence $\mathbfe{\alpha}$. We end up with 
\begin{align*}
&\sum_{A \in \mathcal{M}_{2n}} \Big( \sum_{B \in I(A;h)} S_{n,m; \gamma} (B) \Big)^2 \\
= &\frac{1}{q^{2n-2h+1}} \sum_{ \mathbfe{\alpha} }
	\Bigg( \sum_{\substack{ \mathbf{e} \in \mathbb{F}_q^{n} \times \{ 1 \} \\ \mathbf{f} \in \mathbb{F}_q^{m} \times \{ 1 \} }}
		\psi \bigg( \mathbf{e}^T H_{n+1 , n+1} (\mathbfe{\alpha}) \mathbf{e}
			+ \mathbf{f}^T \gamma H_{m+1 , m+1} (\mathbfe{\alpha}) \mathbf{f}
		\bigg) \Bigg) \\
	& \hspace{6em} \times \Bigg( \sum_{\substack{ \mathbf{e} \in \mathbb{F}_q^{n} \times \{ 1 \} \\ \mathbf{f} \in \mathbb{F}_q^{m} \times \{ 1 \} }}
		\psi \bigg( \mathbf{e}^T H_{n+1 , n+1} (- \mathbfe{\alpha}) \mathbf{e}
			+ \mathbf{f}^T \gamma H_{m+1 , m+1} (- \mathbfe{\alpha}) \mathbf{f}
		\bigg) \Bigg) ,
\end{align*}
where the first sum on the right side is over the restricted values of $\mathbfe{\alpha}$. \\

In order to evaluate the above, it suffices to understand what values $\mathbf{e}^T H_{n+1 , n+1} (\mathbfe{\alpha}) \mathbf{e}$ and $\mathbf{f}^T H_{m+1 , m+1} (\mathbfe{\alpha}) \mathbf{f}$ take in $\mathbb{F}_q$ as $\mathbf{e} , \mathbf{f}$ vary in their respective ranges, and how often these values are attained. \\

We should compare this to the variance of the divisor function, where in \cite{Yiasemides2021_VariCorrDivFuncFpTHankelMatr_ArXiv_v2} the author used the same approach of additive characters and Hankel matrices. There, instead of working with $\mathbf{e}^T H_{n+1 , n+1} (\mathbfe{\alpha}) \mathbf{e}$ and $\mathbf{f}^T H_{m+1 , m+1} (\mathbfe{\alpha}) \mathbf{f}$, we worked with $\mathbf{e}^T H_{m+1 , n+1} (\mathbfe{\alpha}) \mathbf{f}$. Generally, the $i$-th coefficient of $\mathbf{e}$ can take values in $\mathbb{F}_q$, and so considering the sum over $e_i$ only, we have
\begin{align*}
\frac{1}{q} \sum_{e_i \in \mathbb{F}_q} \psi \Big( e_i (R_i \mathbf{f}) \Big)
= \begin{cases}
1 &\text{ if $R_i \mathbf{f} = 0$,} \\
0 &\text{ otherwise;} 
\end{cases}
\end{align*}
where $R_i$ is the $i$-th row of $H_{m+1 , n+1} (\mathbfe{\alpha})$, and the equality follows from (\ref{statement, intro, exp identity sum}). Thus, applying this to all $i$, we see that a non-zero contribution occurs if and only if $\mathbf{f}$ is in the kernel of $H_{m+1 , n+1} (\mathbfe{\alpha})$. Thus, it sufficed to understand the rank of our Hankel matrices. \\

Unfortunately, for our sums-of-squares problem we cannot continue in this manner because in the matrix multiplications $\mathbf{e}^T H_{n+1 , n+1} (\mathbfe{\alpha}) \mathbf{e}$ and $\mathbf{f}^T H_{m+1 , m+1} (\mathbfe{\alpha}) \mathbf{f}$, the vectors on the left and right of the matrix are identical (transposes of each other, technically). That is, we are more-or-less working with quadratic forms. Now, as our Hankel matrices are symmetric square matrices, one may consider diagonalising them. Indeed, the change of basis may preserve the ranges of $\mathbf{e}$ and $\mathbf{f}$, and it is much easier to work with diagonal matrices. However, it is not immediately obvious what eigenvalues (diagonal entries) the Hankel matrices can take, and how many Hankel matrices can be diagonalised to a given diagonal matrix, and so this approach also does not quite work. Instead, we will ``diagonalise'' our Hankel matrices into block-diagonal matrices, where each block is a lower skew-triangular Hankel matrix. This diagonalisation is natural and much easier to understand; the change of variable preserves the ranges of $\mathbf{e}$ and $\mathbf{f}$; and we can consider each block separately, and their lower skew-triangular form makes them relatively easy to work with. \\

However, a significant amount of work is still required, especially when one considers that we must simultaneously work with $H_{n+1 , n+1} (\mathbfe{\alpha})$ and its submatrix $H_{m+1 , m+1} (\mathbfe{\alpha})$, as well as the restrictions imposed on $\mathbfe{\alpha}$. Indeed, Section \ref{section, Hankel matrices} is dedicated to developing further the theory on Hankel matrices to achieve this. \\

Let us now consider how one would extend the approach above to $S_{\gamma} (B)$. Note that $S_{\gamma} (B) = \sum_m S_{\gamma ; m} (B)$. The sum over $m$ can easily be dealt with for the mean, but not for the variance. Indeed, we have $S_{\gamma} (B)^2 = \big( \sum_m S_{\gamma ; m} (B) \big)^2$. We have the ``diagonal'' terms $S_{\gamma ; m_1 } (B) S_{\gamma ; m_2 } (B)$ for $m_1 = m_2$, which we can deal with in the same way as above; but we also have the ``off-diagonal'' terms $S_{\gamma ; m_1 } (B) S_{\gamma ; m_2 } (B)$ for $m_1 \neq m_2$, which are significantly more difficult to work with. It may be possible to calculate these by brute force, but it would be incredibly tedious. Instead, it is more logical to work with all possible values of $m_1 , m_2$ simultaneously, so that averaging and cancellations can take place. This, however, requires us to develop further results on sums of additive characters involving Hankel matrices, and that is why it forms Part II of this paper. \\

Let us now consider the obstacles we encounter if we apply the approach above to the variance of $S_{T} (B)$. Unlike $S_{\gamma ; m} (B)$ and $S_{\gamma} (B)$, which largely follow the same approach, the function $S_{T} (B)$ has some fundamental differences. Instead of working with
\begin{align*}
H_{n+1 , n+1} (\mathbfe{\alpha})
\hspace{3em} \text{ and } \hspace{3em}
H_{m+1 , m+1} (\mathbfe{\alpha}) ,
\end{align*}
we end up working with
\begin{align*}
H_{n+1 , n+1} (\mathbfe{\alpha})
\hspace{3em} \text{ and } \hspace{3em}
H_{m+1 , m+1} (\mathbfe{\alpha}') ,
\end{align*}
where $\mathbfe{\alpha}' := (\alpha_1 , \alpha_2 , \ldots , \alpha_{2m+1})$. For intuition, note that $H_{m+1 , m+1} (\mathbfe{\alpha})$ is the top-left\break $(m+1) \times (m+1)$ submatrix of $H_{n+1 , n+1} (\mathbfe{\alpha})$, whereas $H_{m+1 , m+1} (\mathbfe{\alpha}')$ is the top-left\break $(m+1) \times (m+1)$ submatrix of the matrix obtained by removing the first column of $H_{n+1 , n+1} (\mathbfe{\alpha})$. That is, $H_{m+1 , m+1} (\mathbfe{\alpha}')$ is ``shifted'' one column to the right, which ultimately stems from the fact that multiplying $F^2$ by $T$ ``shifts'' the coefficients of $F^2$ one space to the right. \\

Now, the entries in $H_{m+1 , m+1} (\mathbfe{\alpha})$ form the sequence $\mathbfe{\alpha}'' := (\alpha_0 , \alpha_1 , \ldots , \alpha_{2m})$, which is a subsequence of $\mathbfe{\alpha}$. In particular, one can obtain $\mathbfe{\alpha}$ by extending $\mathbfe{\alpha}''$ to the right. Whereas, the sequence $\mathbfe{\alpha}'$ is a subsequence of $\mathbfe{\alpha}$, but we must extend $\mathbfe{\alpha}'$ to the left and right in order to obtain $\mathbfe{\alpha}$. It is much easier to understand the relationship between Hankel matrices associated to sequences where one sequence is an extension of the other in one direction only. Whereas, if the extension is in two directions it is considerably more difficult. In fact, it requires a deeper understanding of Hankel matrices and their kernel structure. This, as well as the application to the variance of $S_{T} (B)$, is an interesting problem that warrants further investigation.

\section{Hankel Matrices} \label{section, Hankel matrices}

In this section we prove some results on Hankel matrices over $\mathbb{F}_q$ that we will require. Section 2 of \cite{Yiasemides2021_VariCorrDivFuncFpTHankelMatr_ArXiv_v2} develops much of the theory of these matrices, and we will use some results from there. We will be concerned with the square matrix case. So, suppose $H$ is an $n \times n$ Hankel matrix. \\

In Subsection 2.1 of \cite{Yiasemides2021_VariCorrDivFuncFpTHankelMatr_ArXiv_v2}, based on the analogous definition in \cite{HeinigRost1984_AlgMethToeplitzMatrOperat} for real Hankel matrices, they define the \rhopi-characteristic of $H$ to be the pair $(\rho_1 , \pi_1 )$, where $\rho_1$ is the largest integer in $\{ 1 , \ldots , n \}$ such that $H [ \rho_1 , \rho_1 ]$ is invertible, or if no such integer exists then $\rho_1 := 0$; and $\pi_1 := \rank (H) - \rho_1$. They write $\rho (H) = \rho_1$ and $\pi (H) = \pi_1$. \\

Here, we define the strict \rhopi-characteristic in the same way, except $\rho_1$ is the largest integer in $\{ 1 , \ldots , n-1 \}$ such that $H [ \rho_1 , \rho_1 ]$ is invertible, or if no such integer exists then $\rho_1 := 0$ (the difference is that $\rho_1$ cannot take the value $n$). We still take $\pi_1 := \rank (H) - \rho_1$. We write $\rho_s (H) = \rho_1$ and $\pi_s (H) = \pi_1$. \\

In Subsection 2.2 of \cite{Yiasemides2021_VariCorrDivFuncFpTHankelMatr_ArXiv_v2}, they define the \rhopi-form of Hankel matrices. In this paper we will define the reduced \rhopi-form of a square Hankel matrix which is, as the name suggests, related to the \rhopi-form. It should be noted, that in \cite{Yiasemides2021_VariCorrDivFuncFpTHankelMatr_ArXiv_v2} they work with the \rhopi-characteristic, while in this paper we work with the strict \rhopi-characteristic. However, this will not cause any abuse of definitions. 

\subsection{The Reduced $(\rho , \pi )$-form} \label{subsection, reduced rho pi form}

We proceed in a similar manner as in Section 2.2 of \cite{Yiasemides2021_VariCorrDivFuncFpTHankelMatr_ArXiv_v2}. If the strict \rhopi-characteristic of $H$ is $(0,0)$, then $H$ is simply the zero matrix; while if the strict \rhopi-characteristic of $H$ is $(0 , \pi_1 )$ for some $1 \leq \pi_1 \leq n$, then $H$ is of the form
\begin{align*}
\begin{pNiceMatrix}
0 & \Cdots & \Cdots & \Cdots & \Cdots & \Cdots & 0 \\
\Vdots &  &  &  &  &  & \Vdots \\
0 & \Cdots & \Cdots & \Cdots & \Cdots & \Cdots & 0 \\
0 & \Cdots & \Cdots & \Cdots & \Cdots & 0 & \lambda \\
0 & \Cdots & \Cdots & \Cdots & 0 & \lambda & \beta_{2} \\
\Vdots &  &  & \Iddots & \Iddots & \Iddots & \Vdots \\
0 & \Cdots & 0 & \lambda & \beta_{2} & \Cdots & \beta_{\pi_1 } 
\end{pNiceMatrix} ,
\end{align*}
where $\lambda \in \mathbb{F}_q^*$ and $\beta_{2} , \ldots , \beta_{\pi_1 }  \in \mathbb{F}_q$. (If $\pi_1 = n$, then the above should be interpreted as not having any rows/columns of zero at the top/left). A very similar result was proven in Subsection 2.2 of \cite{Yiasemides2021_VariCorrDivFuncFpTHankelMatr_ArXiv_v2}. Now, in these cases we simply define the reduced \rhopi-form of $H$ to be $H$. \\

Now suppose that the strict \rhopi-characteristic of $H = (\alpha_{i+j-2} )_{1 \leq i,j \leq n}$ is $(\rho_1 , \pi_1 )$ where $1 \leq \rho_1 \leq n-1$. Consider the submatrix $H [\rho_1 , \rho_1 ]$ and note that, by definition of $\rho_1$, it is invertible. In particular, there exists a unique solution $\mathbf{x} = (x_0 , \ldots , x_{\rho_1 -1})^T$ to the following equation:
\begin{align*}
H [\rho_1 , \rho_1 ] \mathbf{x}
= \begin{pmatrix} \alpha_{\rho_1 } \\ \vdots \\ \alpha_{2 \rho_1 -1} \end{pmatrix} .
\end{align*}
The vector on the right side is simply the column that appears to the right of $H [\rho_1 , \rho_1 ]$ in $H$ (equivalently, its transpose is the row that appears below $H [\rho_1 , \rho_1 ]$). It should be noted that $\mathbf{x}$ is uniquely determined by $H [\rho_1 , \rho_1 ]$ and the vector $(\alpha_{\rho_1 } , \ldots , \alpha_{2 \rho_1 -1})^T$. Since $\alpha_{\rho_1 } , \ldots , \alpha_{2 \rho_1 -2}$ appear in the matrix $H [\rho_1 , \rho_1 ]$, we can see that $\mathbf{x}$ is uniquely determined by $H [\rho_1 , \rho_1 ]$ and $\alpha_{2 \rho_1 -1}$. \\

Now, let $R_1 , \ldots , R_n$ be the rows of $H$, and let us apply the following row operations:
\begin{align*}
R_i \longrightarrow R_i - (x_0 , \ldots , x_{\rho_1 -1}) \begin{pmatrix} R_{i - \rho_1 } \\ \vdots \\ R_{i-1} \end{pmatrix} 
	= R_i - x_0 R_{i - \rho_1 } - \ldots - x_{\rho_1 -1} R_{i-1} 
\end{align*}
for $i = n , n -1 , \ldots , \rho_1 +1$ in that order. We then obtain a matrix of the form
\begin{align*}
\begin{pmatrix}
H [\rho_1 , \rho_1] & \vline & H [\rho_1 , -(m-\rho_1)] \\
\hline
\mathbf{0} & \vline & 
\begin{NiceMatrix}
0 & \Cdots & \Cdots & \Cdots & \Cdots & \Cdots & \Cdots & 0 \\
\Vdots &  &  &  &  &  &  & \Vdots \\
0 & \Cdots & \Cdots & \Cdots & \Cdots & \Cdots & \Cdots & 0 \\
0 & \Cdots & \Cdots & \Cdots & \Cdots & \Cdots & 0 & \lambda_1 \\
0 & \Cdots & \Cdots & \Cdots & \Cdots & 0 & \lambda_1 & \beta_2 ' \\
0 & \Cdots & \Cdots & \Cdots & 0 & \lambda_1 & \beta_2 ' & \beta_3 ' \\
\Vdots &  &  & \Iddots & \Iddots & \Iddots & \Iddots & \Vdots \\
0 & \Cdots & 0 & \lambda_1 & \beta_2 ' & \beta_3 ' & \Cdots & \beta_{\pi_1} ' 
\end{NiceMatrix}
\end{pmatrix} ,
\end{align*}
where $\lambda_1 \in \mathbb{F}_q^*$ and $\beta_2 ' , \ldots , \beta_{\pi_1 } ' \in \mathbb{F}_q$. An almost identical result was proven in Subsection 2.2 of \cite{Yiasemides2021_VariCorrDivFuncFpTHankelMatr_ArXiv_v2}. The only difference being that in \cite{Yiasemides2021_VariCorrDivFuncFpTHankelMatr_ArXiv_v2} they worked with the \rhopi-characteristic instead of the strict \rhopi-characteristic. \\

Now, let $C_1 , \ldots , C_n$ be the rows of $H$, and let us apply the analogous column operations:
\begin{align*}
C_i \longrightarrow C_i - (x_0 , \ldots , x_{\rho_1 -1}) \begin{pmatrix} C_{i - \rho_1 } \\ \vdots \\ C_{i-1} \end{pmatrix} 
	= C_i - x_0 C_{i - \rho_1 } - \ldots - x_{\rho_1 -1} C_{i-1} 
\end{align*}
for $i = n , n -1 , \ldots , \rho_1 +1$ in that order. By similar reasoning as above, we obtain a matrix of the form
\begin{align*}
\begin{pmatrix}
H [\rho_1 , \rho_1] & \vline & \mathbf{0} \\
\hline
\mathbf{0} & \vline & 
\begin{NiceMatrix}
0 & \Cdots & \Cdots & \Cdots & \Cdots & \Cdots & \Cdots & 0 \\
\Vdots &  &  &  &  &  &  & \Vdots \\
0 & \Cdots & \Cdots & \Cdots & \Cdots & \Cdots & \Cdots & 0 \\
0 & \Cdots & \Cdots & \Cdots & \Cdots & \Cdots & 0 & \lambda_1 \\
0 & \Cdots & \Cdots & \Cdots & \Cdots & 0 & \lambda_1 & \beta_2 \\
0 & \Cdots & \Cdots & \Cdots & 0 & \lambda_1 & \beta_2 & \beta_3 \\
\Vdots &  &  & \Iddots & \Iddots & \Iddots & \Iddots & \Vdots \\
0 & \Cdots & 0 & \lambda_1 & \beta_2 & \beta_3 & \Cdots & \beta_{\pi_1} 
\end{NiceMatrix}
\end{pmatrix} ,
\end{align*}
where $\beta_2 , \ldots , \beta_{\pi_1 } \in \mathbb{F}_q$. We can express the above in form
\begin{align} \label{statement, reduce rho pi form section, matrix after first iteration}
\begin{pmatrix}
H [\rho_1 , \rho_1] & \vline & \mathbf{0} \\
\hline
\mathbf{0} & \vline & H_1
\end{pmatrix}
=
\begin{pmatrix}
H [\rho_1 , \rho_1] & \vline & \mathbf{0} & \vline & \mathbf{0} \\
\hline
\mathbf{0} & \vline & H_1 '' & \vline & \mathbf{0} \\
\hline
\mathbf{0} & \vline & \mathbf{0} & \vline & H_1 ' \\
\end{pmatrix} ,
\end{align}
where
\begin{align*}
H_1 '
= \begin{pNiceMatrix}
0 & \Cdots & \Cdots & 0 & \lambda_1 \\
0 & \Cdots & 0 & \lambda_1 & \beta_2 \\
\Vdots & \Iddots & \Iddots & \Iddots & \Vdots \\
0 & \lambda_1 & \beta_2 & \Cdots & \beta_{\pi_1 -1} \\
\lambda_1 & \beta_2 & \Cdots & \Cdots & \beta_{\pi_1} 
\end{pNiceMatrix} ,
\end{align*}
$H_1 ''$ is the $(n-r) \times (n-r)$ zero matrix (where $r = \rank (H)$), and
\begin{align*}
H_1
= \begin{pmatrix}
H_1 '' & \vline & \mathbf{0} \\
\hline
\mathbf{0} & \vline & H_1 '
\end{pmatrix} .
\end{align*}
It may be helpful to recall that $r = \rank (H) = \rho_1 + \pi_1$ (by definition of the strict \rhopi-characteristic) and that is why the number of rows/columns of $H''$ is $n- \rho_1 - \pi_1 = n-r$. It is important to note for later than $H_1 '$ is a lower skew-triangular square Hankel matrix with non-zero main skew-diagonal, which we will simply refer to as a non-strict lower skew-triangular Hankel matrix. Of course, in the special case where $H$ has full rank, the matrix $H_1 ''$ should be ignored; and in the special case where $\pi_1 = 0$, the matrix $H_1 '$ should be ignored. This can be achieved by viewing them as the ``$0 \times 0$ empty matrix'', which is consistent with their defined dimensions. \\

Now, we can apply the above to the matrix $H [\rho_1 , \rho_1]$ instead of $H$. In fact, we can continue in an inductive manner to ultimately obtain a matrix of the form
\begin{align} \label{statement, ES rho pi form of H}
\begin{pNiceMatrix}
H_t & \mathbf{0} & \Cdots &  &  & \mathbf{0} \\
\mathbf{0} & H_{t-1} & \Ddots &  &  & \Vdots \\
\Vdots & \Ddots & \Ddots &  &  &  \\
 &  &  & H_2 & \mathbf{0} & \mathbf{0} \\
 &  &  & \mathbf{0} & H_1 '' & \mathbf{0} \\
\mathbf{0} & \Cdots &  & \mathbf{0} & \mathbf{0} & H_1 '
\end{pNiceMatrix} ,
\end{align}
where $H_1'$ and $H_1''$ are as described previously, and $H_2 , \ldots , H_t$ are non-strict lower skew-triangular Hankel matrices. 

\begin{definition}[Reduced \rhopi-form and Stage 1 \rhopi-form]
We define the reduced \rhopi-form of $H$ to be the matrix (\ref{statement, ES rho pi form of H}). This applies when $\rho_s (H) \geq 1$. When $\rho_s (H) = 0$, we define the reduced \rhopi-form of $H$ to be $H$, as mentioned at the start of this subsection. We denote the reduced \rhopi-form by $\reduced (H)$. \\

It is also necessary at times to work with (\ref{statement, reduce rho pi form section, matrix after first iteration}), which is the first stage in our process for obtaining the reduced \rhopi-form. Thus, we define the Stage 1 \rhopi-form of $H$ to be (\ref{statement, reduce rho pi form section, matrix after first iteration}). Again, this applies when $\rho_s (H) \geq 1$. When $\rho_s (H) = 0$, we simply take the Stage 1 \rhopi-form to be $H$. We denote the Stage 1 \rhopi-form by $\reduced_1 (H)$.
\end{definition}

\begin{definition}[Reduced Matrix]
A reduced \rhopi-form matrix is a matrix of the form (\ref{statement, ES rho pi form of H}).\footnote{That is, a block-diagonal square matrix, where each block is a non-strict lower skew-triangular Hankel matrix, except for $H_1 ''$ which is the zero matrix. Naturally, $t$ can take any positive integer value (if $t=1$, then there are no blocks $H_2 , H_3 , \ldots$); the number of rows/columns of $H_2 , \ldots , H_t$ can each take any positive integer value; and the number of rows/columns of $H_1 '$ and $H_1 ''$ can each take any non-negative integer value, but at least one must be positive.} For convenience, in this paper where it will not be confused with (for example) reduced row echelon matrices, we will simply refer to such matrices as reduced matrices. We denote the set of all $n \times n$ reduced matrices by $\mathscr{R}_n$.
\end{definition}

\begin{definition}[\rhopi-partition] \label{definition, rho pi partition}
Suppose $M$ is a reduced matrix of the form (\ref{statement, ES rho pi form of H}), and let
\begin{align*}
p_1 ' , p_1 '' , p_2 , \ldots , p_t
\end{align*}
be the number of rows/columns of $H_1 ' , H_1 '' , H_2 , \ldots , H_t$. We then define the \rhopi-partition of $M$ to be
\begin{align*}
\partition (M) 
:= (p_1 ' , p_1 '' , p_2 , \ldots , p_t ) .
\end{align*}
Now suppose we have a square Hankel matrix $H$. We define the \rhopi-partition of $H$ to be the \rhopi-partition of $\reduced (H)$. We similarly denote this as $\partition (H)$. \\

Note that $\rho_s (H) = p_2 + \ldots p_t$ and $\pi_s (H) = p_1'$, and $\rank (H) = p_1' + p_2 + \ldots p_t$. \\

We denote the set of all possible \rhopi-partitions associated to $n \times n$ Hankel matrices by $\mathscr{P}_n$. That is,
\begin{align*}
\mathscr{P}_n
:= \{ \partition (H) : H \in \mathscr{H}_{n,n} \} .
\end{align*}
\end{definition}

\subsection{Counting Matrices}

In this subsection, we are interested in the number of Hankel matrices that have a given reduced \rhopi-form, and the number of reduced matrices that have a given \rhopi-partition. Thus, let us make the following definitions.

\begin{definition}
For a \rhopi-partition $P \in \mathscr{P}_n$, we define
\begin{align*}
S_{\reduced } (P)
:= \{ M \in \mathscr{R}_n : \partition (M) = P \} ,
\end{align*}
and
\begin{align*}
S_{\hank } (P)
:= \{ H \in \mathscr{H}_{n,n} : \partition (H) = P \} .
\end{align*}

For a reduced matrix $M \in \mathscr{R}_n$, we define 
\begin{align*}
S_{\hank } (M)
:= \{ H \in \mathscr{H}_{n,n} : \reduced (H) = M \} .
\end{align*}
\end{definition}

\begin{lemma} \label{lemma, number of reduced matrices with given partition}
Suppose we have a \rhopi-partition $P = (p_1 ' , p_1 '' , p_2 , \ldots , p_t )$, and let $r = p_1 ' + p_2 + \ldots + p_t$. Then,
\begin{align*}
\lvert S_{\reduced } (P) \rvert
= (q-1)^t q^{r-t} .
\end{align*}
\end{lemma}

\begin{proof}
Suppose we have a reduced matrix $M$ with $\partition (M) = P$. Then, $M$ has the form
\begin{align*}
\begin{pNiceMatrix}
H_t & \mathbf{0} & \Cdots &  &  & \mathbf{0} \\
\mathbf{0} & H_{t-1} & \Ddots &  &  & \Vdots \\
\Vdots & \Ddots & \Ddots &  &  &  \\
 &  &  & H_2 & \mathbf{0} & \mathbf{0} \\
 &  &  & \mathbf{0} & H_1 '' & \mathbf{0} \\
\mathbf{0} & \Cdots &  & \mathbf{0} & \mathbf{0} & H_1 '
\end{pNiceMatrix} .
\end{align*}
There is only one value that the zero matrix $H_1 ''$ can take. As for the other blocks, we must determine how many non-strict lower skew-symmetric matrices there are of a given size, say $l \times l$. Clearly, this is $(q-1) q^{l-1}$, because the main skew-diagonal can take values in $\mathbb{F}_q^*$, while the $l-1$ skew-diagonals below that can take values in $\mathbb{F}_q$. From this, we determine that our answer is
\begin{align*}
\lvert S_{\reduced } (P) \rvert 
= (q-1)^t q^{r-t} .
\end{align*}
\end{proof}

\begin{lemma} \label{lemma, number of Hankel matrices with given reduced matrix}
Suppose we have a reduced matrix $M$ with \rhopi-partition equal to \\$(p_1 ' , p_1 '' , p_2 , \ldots , p_t )$. Then,
\begin{align*}
\lvert S_{\hank } (M) \rvert
= q^{t-1} .
\end{align*}
\end{lemma}

\begin{proof}
Let $H = (\alpha_{i+j-2} )_{1 \leq i , j \leq n}$ be a Hankel matrix such that $\reduced (H) = M$. We must show that there are $q^{t-1}$ possible values that $H$ can take. To prove this, we must look back at how we obtain $\reduced (H)$ from $H$. \\

Consider the matrix (\ref{statement, reduce rho pi form section, matrix after first iteration}) which we obtain after applying row and column operations on $H$. There are exactly $q$ values that $H$ can take that will leave us with (\ref{statement, reduce rho pi form section, matrix after first iteration}) after applying the rows and column operations. To see this, first note that $H [\rho_1 , \rho_1 ]$ is also the top-left $\rho_1 \times \rho_1$ submatrix of $H$, and so this determines the values of $\alpha_0 , \alpha_1 , \ldots , \alpha_{2 \rho_1 - 2}$. The value of $\alpha_{2 \rho_1 -1}$ is undetermined and free to take any value in $\mathbb{F}_q$. Recall that $\mathbf{x}$ is determined by $H [\rho_1 , \rho_1 ]$ and $\alpha_{2 \rho_1 -1}$; and it is not difficult to see from our row and column operations that $\alpha_{2 \rho_1} , \ldots , \alpha_{2n-2}$ are determined by $\mathbf{x}$, $H [\rho_1 , \rho_1 ]$, $\alpha_{2 \rho_1 -1}$, and $H_1 '$ and $H_1 ''$. Thus, it is only $\alpha_{2 \rho_1 -1}$ that gives us a degree of freedom, with $q$ possible values that it could take. \\

Now, obtaining (\ref{statement, reduce rho pi form section, matrix after first iteration}) is only the first step in our procedure of obtaining $\reduced (H)$ from $H$. Indeed, we apply similar steps a total number of $t-1$ times. Thus, it is not difficult to deduce that there $q^{t-1}$ possible values that $H$ could take.
\end{proof}

\subsection{Multisets}

We will need to introduce some notation. A multiset is similar to a set, but repetition of elements is ``allowed''. For example, $\{ 1 , 2 \}$ and $\{ 1 , 2 , 2 \}$ are the same sets, but when viewed as multisets, they are different. To distinguish we use the conventional notation of square brackets to represent multisets. For example, $[ 1 , 2 ]$ and $[1 , 2 , 2]$. In this paper, all multisets that we work with are finite. \\

The multiplicity of an element in a multiset is defined to be the number of times it appears in that multiset. For a multiset $A$, we denote the multiplicity of $a$ in $A$ by $m_A (a)$. Obviously, if $a$ does not appear in $A$, then we take $m_A (a) = 0$. The function $m_A$ is called the multiplicity function of $A$. \\

In this paper, regarding the union of two multisets $A,B$, we define $A \cup B$ to be the multiset $C$ define by the multiplicity function $m_C := m_A + m_B$. That is, if $a$ appears $n_1$ times in $A$ and $n_2$ times in $B$, then $a$ appears $n_1 + n_2$ times in $C$. This is non-standard terminology and notation. Our definition of the union of $A$ and $B$ is what is typically referred to as the sum of $A$ and $B$. However, we will be working with sumsets of multisets, and so we must reserve ``sum'' and ``$+$'' for that. \\

We use the following notation for multisets that we will use often. We define $F_q$ to be the multiset where each element of $\mathbb{F}_q$ appears exactly once. We define $T_q$ to be the multiset where $0$ appears $2q-1$ times and each element of $\mathbb{F}_q^*$ appears $q-1$ times. For positive integers $n$, we define $Z_q (n)$ to be the multiset where $0 \in \mathbb{F}_q$ appears $n$ times. Finally, for $\lambda \in \mathbb{F}_q$, we define $S_q (\lambda ) := [ \lambda a^2 : a \in \mathbb{F}_q ]$. \\

Now let us define the sumset. For subsets $A,B$ of an abelian group, we define the sum set
\begin{align*}
A + B = \{ a+b : a \in A , b \in B \} . 
\end{align*}
For multisets $A,B$, whose elements lie in an abelian group, we define the sumset to be the multiset
\begin{align*}
A + B = [ a+b : a \in A , b \in B ] . 
\end{align*}
Of course, multiplicity is allowed in the ranges $a \in A , b \in B$, and multiplicity is counted for the values $a+b$. That is, $m_{A+B} (c) = \sum_{a+b=c} m_A (a) m_B (b)$. \\

For non-negative integers $n$, we define 
\begin{align*}
nA := \underbrace{A+A+ \ldots + A}_{\text{$A$ appears $n$ times here}}
\end{align*}
and we take the convention that $0A = [0]$. \\

Now, let $H$ be an $n \times n$ Hankel matrix. In this section, we wish to understand, for each $a \in \mathbb{F}_q$, how many $\mathbf{v} \in \mathbb{F}_q^{n-1} \times \{ 1 \}$ there are such that $\mathbf{v}^T H \mathbf{v} = a$. This leads us to define the following three multisets. \\

First, though, for a non-zero vector $\mathbf{w}$ over $\mathbb{F}_q$, let us define $l (\mathbfe{w})$ to be the last non-zero entry of $\mathbf{w}$. That is, if $\mathbf{w}^T = (w_1 , \ldots , w_k , 0 \ldots , 0)$ with $w_k \neq 0$, then $l(\mathbf{w}) = w_k$. If $\mathbf{w}$ is a zero vector, then we define $l (\mathbf{w}) = 0$. \\

\begin{definition}
For an $n \times n$ matrix $M$, we define the multisets
\begin{align*}
\multisetmatrix (M)
:= \Big[ \mathbf{v}^T M \mathbf{v} : \mathbf{v} \in \mathbb{F}_q^n \Big]
\end{align*}
and
\begin{align*}
\multisetmatrix_{\mathcal{M}} (M)
:= \Big[ \mathbf{v}^T M \mathbf{v} : \mathbf{v} \in \mathbb{F}_q^{n-1} \times \{ 1 \} \Big] 
\end{align*}
and
\begin{align*}
\multisetmatrix_{1} (M)
:= \Big[ \mathbf{v}^T M \mathbf{v} : \mathbf{v} \in \mathbb{F}_q^{n} , l (\mathbf{v}) = 1 \Big]
= \bigcup_{i=1}^{n} \multisetmatrix_{\mathcal{M}} \big( M [i;i] \big) .
\end{align*}
\end{definition}

\begin{lemma} \label{lemma, multiset H equals multiset reduced H}
For a square Hankel matrix $H$, we have that
\begin{align*}
\multisetmatrix (H)
= \multisetmatrix \Big( \reduced (H) \Big)
\end{align*}
and
\begin{align*}
\multisetmatrix_{\mathcal{M}} (H)
= \multisetmatrix_{\mathcal{M}} \Big( \reduced (H) \Big)
\end{align*}
and
\begin{align*}
\multisetmatrix_{1} (H)
= \multisetmatrix_{1} \Big( \reduced (H) \Big) .
\end{align*}
\end{lemma}

\begin{proof}
We obtained $\reduced (H)$ from $H$ by applying row and column operations. Let $P_i$ be the matrix associated with the $i$-th row operation that we applied, and suppose that there were $k$ operations in total. Recall that column operations were the exact analogue of the row operations. In particular, we have
\begin{align*}
\reduced (H) = (P_k P_{k-1} \ldots P_1 ) H (P_k P_{k-1} \ldots P_1 )^T .
\end{align*}
We can write the above as
\begin{align*}
H = P \reduced (H) P^T ,
\end{align*}
where $P = (P_k P_{k-1} \ldots P_1 )^{-1}$. \\

So, we have that 
\begin{align*}
\mathbf{v}^T H \mathbf{v}
= (P^T \mathbf{v})^T \reduced (H) (P^T \mathbf{v}) .
\end{align*}
Since $P$ is invertible, we can easily deduce that
\begin{align*}
\{ P^T \mathbf{v} : \mathbf{v} \in \mathbb{F}_q^n \}
= \mathbb{F}_q^n .
\end{align*}
Thus, 
\begin{align*}
\multisetmatrix (H)
= &\Big[ \mathbf{v}^T H \mathbf{v} : \mathbf{v} \in \mathbb{F}_q^n \Big]
= \Big[ (P^T \mathbf{v})^T \reduced (H) (P^T \mathbf{v}) : \mathbf{v} \in \mathbb{F}_q^n \Big] \\
= &\Big[ \mathbf{v}^T \reduced (H) \mathbf{v} : \mathbf{v} \in \mathbb{F}_q^n \Big]
= \multisetmatrix \Big( \reduced (H) \Big) ,
\end{align*}
which proves the first claim of the Lemma. \\

The proof of the second claim follows in a similar manner. We have
\begin{align*}
\multisetmatrix_{\mathcal{M}} (H)
= &\Big[ \mathbf{v}^T H \mathbf{v} : \mathbf{v} \in \mathbb{F}_q^{n-1} \times \{ 1 \} \Big]
= \Big[ (P^T \mathbf{v})^T \reduced (H) (P^T \mathbf{v}) : \mathbf{v} \in \mathbb{F}_q^{n-1} \times \{ 1 \} \Big] \\
= &\Big[ \mathbf{v}^T \reduced (H) \mathbf{v} : \mathbf{v} \in \mathbb{F}_q^{n-1} \times \{ 1 \} \Big]
= \multisetmatrix_{\mathcal{M}} \Big( \reduced (H) \Big) ,
\end{align*}
where we have used the fact that 
\begin{align*}
\{ P^T \mathbf{v} : \mathbf{v} \in \mathbb{F}_q^{n-1} \times \{ 1 \} \}
= \mathbb{F}_q^{n-1} \times \{ 1 \} .
\end{align*}
Unlike previously, this does not follow only from the fact that $P$ is invertible, as it is not immediately obvious that $P^T \mathbf{v}$ will have a $1$ in its final row simply because $\mathbf{v}$ has a $1$ in the final row. Indeed, this is something we must prove. \\

To do this, recall that $P = (P_k P_{k-1} \ldots P_1 )^{-1}$, and that $P_1$ is the only matrix factor in this product that is associated to a row operation on the final row. That is, it is $P_1$ that determines whether the $P^T \mathbf{v}$ will have a $1$ in its final row when $\mathbf{v}$ does. Now, we have that
\begin{align*}
P_1
= \begin{pmatrix}
I_{n-1} & \vline & \begin{matrix} 0 \\ \vdots \\ 0 \end{matrix} \\
\hline
\begin{matrix} 0 & \cdots & 0 & -x_0 & \cdots & -x_{\rho_1 -1} \end{matrix} & \vline & 1 
\end{pmatrix} ,
\end{align*}
where $\mathbf{x} = ( x_0 , \ldots , x_{\rho_1 -1})^T$ is as in Subsection \ref{subsection, reduced rho pi form}. It is not difficult to now see that $(P_1^{-1} )^T \mathbf{v}$ has a $1$ in the final row when $\mathbf{v}$ does. \\

For the third claim, we have
\begin{align*}
\multisetmatrix_{1} (H)
= &\bigcup_{i=1}^{n} \multisetmatrix_{\mathcal{M}} \big( H [i;i] \big) 
= \bigcup_{i=1}^{n} \multisetmatrix_{\mathcal{M}} \Big( \reduced \big( H [i;i] \big) \Big) \\
= &\bigcup_{i=1}^{n} \multisetmatrix_{\mathcal{M}} \Big( \reduced (H) [i;i] \Big) 
= \multisetmatrix_{1} \Big( \reduced (H) \Big) .
\end{align*}
The first and last equality follow from the fact that $\multisetmatrix_{1} (M) = \bigcup_{i=1}^{n} \multisetmatrix_{\mathcal{M}} \big( M [i;i] \big)$ which we established in the definition of $\multisetmatrix_{1} (M)$. The second equality is simply an application of the second claim. The third equality uses the fact that $\reduced \big( H [i;i] \big) = \reduced (H) [i;i]$, which is not difficult to see from how we obtain the reduced \rhopi-form.
\end{proof}

\begin{lemma} \label{lemma, multiset of nonstrict lower skew triang matr}
Let $M$ be an $l \times l$ non-strict lower skew-triangular Hankel matrix. That is,
\begin{align*}
M
= \begin{pNiceMatrix}
0 & \Cdots & \Cdots & \Cdots & 0 & \lambda \\
\Vdots &  &  & \Iddots & \lambda & \alpha_2 \\
\Vdots &  & \Iddots & \Iddots & \alpha_2 & \alpha_3 \\
\Vdots & \Iddots & \Iddots & \Iddots & \Iddots & \Vdots \\
0 & \lambda & \alpha_2 & \Iddots &  & \Vdots \\
\lambda & \alpha_2 & \alpha_3 & \Cdots & \Cdots & \alpha_l 
\end{pNiceMatrix} ,
\end{align*}
where $\lambda \in \mathbb{F}_q^*$ and $\alpha_2 , \ldots , \alpha_l \in \mathbb{F}_q$. Then,
\begin{align*}
\multisetmatrix (M )
= \begin{cases}
\frac{l}{2} T_q &\text{if $l$ is even,} \\
\frac{l-1}{2} T_q + S_q (\lambda ) &\text{if $l$ is odd;}
\end{cases}
\end{align*}
and
\begin{align*}
\multisetmatrix_{\mathcal{M}} (M)
= \begin{cases}
F_q + \frac{l-2}{2} T_q &\text{if $l$ is even,} \\
F_q + \frac{l-3}{2} T_q + S_q (\lambda ) &\text{if $l \neq 1$ is odd,} \\
[ \lambda ] &\text{if $l = 1$.}
\end{cases}
\end{align*}
\end{lemma}

\begin{proof}
Let us write $\hat{M}$ for the matrix obtained by removing the last row and column from $M$. We can apply row and column operations to $M$ to obtain a matrix of the form
\begin{align*}
\begin{pNiceMatrix}
0 & \Cdots & \Cdots & \Cdots & 0 & \lambda \\
\Vdots &  &  & \Iddots & \lambda & 0 \\
\Vdots &  & \Iddots & \Iddots & \alpha_2 & 0 \\
\Vdots & \Iddots & \Iddots & \Iddots & \Vdots & \Vdots \\
0 & \lambda & \alpha_2 & \Cdots & \alpha_{l-1} & 0 \\
\lambda & 0 & 0 & \Cdots & 0 & \alpha_l ' 
\end{pNiceMatrix} 
= \begin{pmatrix}
0 & \vline & \mathbf{0} & \vline & \lambda \\
\hline
\mathbf{0} & \vline & \hat{M} & \vline & \mathbf{0} \\
\hline
\lambda & \vline & \mathbf{0} & \vline & \alpha_l '
\end{pmatrix} 
=: M' ,
\end{align*}
for some $\alpha_l ' \in \mathbb{F}_q$. That is, there is an invertible matrix $P$ such that
\begin{align*}
P M P
= M' .
\end{align*}
Similarly as in the proof of Lemma \ref{lemma, multiset H equals multiset reduced H}, the matrix $P$ is a bijection from $\mathbb{F}_q^l$ to $\mathbb{F}_q^l$ (as well as a bijection from $\mathbb{F}_q^{l-1} \times \{ 1 \}$ to $\mathbb{F}_q^{l-1} \times \{ 1 \}$), and so
\begin{align*}
\multisetmatrix (M )
= \multisetmatrix (M') .
\end{align*}
Also, it is not difficult to see that
\begin{align*}
\multisetmatrix (M')
= \multisetmatrix ( \hat{M} )
	+ \multisetmatrix \Bigg( \begin{pmatrix} 0 & \lambda \\ \lambda & \alpha_l ' \end{pmatrix} \Bigg) .
\end{align*}

Note that
\begin{align*}
&\multisetmatrix \Bigg( \begin{pmatrix} 0 & \lambda \\ \lambda & \alpha_l ' \end{pmatrix} \Bigg) \\
= &\bigg[ \mathbf{v}^T \begin{pmatrix} 0 & \lambda \\ \lambda & \alpha_l ' \end{pmatrix} \mathbf{v} : \mathbf{v} \in \mathbb{F}_q^2 \bigg]
= [2 \lambda v_1 v_2 + \alpha_l ' {v_2}^2 : v_1 , v_2 \in \mathbb{F}_q ]
= [v_1 v_2  : v_1 , v_2 \in \mathbb{F}_q ] 
= T_q ,
\end{align*}
where the second equality follows from the substitution $v_1 \mapsto \frac{1}{2 \lambda } (v_1 - \alpha_1 ' v_2 )$ which is bijective. \\

 So, an inductive argument gives us
\begin{align*}
\multisetmatrix (M )
= \begin{cases}
\frac{l}{2} T_q &\text{ if $l$ is even,} \\
\frac{l-1}{2} T_q + \multisetmatrix \big( ( \lambda ) \big) = \frac{l-1}{2} T_q + S_q (\lambda ) &\text{ if $l$ is odd.}
\end{cases}
\end{align*}

Now consider the second claim of the Lemma. The case $l=1$ is trivial. For the other cases, similarly as above, we have
\begin{align*}
\multisetmatrix_{\mathcal{M}} (M )
= \multisetmatrix ( \hat{M} )
	+ \multisetmatrix_{\mathcal{M}} \Bigg( \begin{pmatrix} 0 & \lambda \\ \lambda & \alpha_l ' \end{pmatrix} \Bigg) .
\end{align*}
We have that
\begin{align*}
&\multisetmatrix_{\mathcal{M}} \Bigg( \begin{pmatrix} 0 & \lambda \\ \lambda & \alpha_l ' \end{pmatrix} \Bigg) \\
= &\bigg[ \mathbf{v}^T \begin{pmatrix} 0 & \lambda \\ \lambda & \alpha_l ' \end{pmatrix} \mathbf{v} : \mathbf{v} \in \mathbb{F}_q \times \{ 1 \} \bigg]
= [2 \lambda v_1 + \alpha_l ' : v_1 \in \mathbb{F}_q ]
= [v_1  : v_1 \in \mathbb{F}_q ] 
= F_q ,
\end{align*}
where the second equality follows from the substitution $v_1 \mapsto \frac{1}{2 \lambda } (v_1 - \alpha_1 ' )$ which is bijective. The remainder of the proof is similar to the proof of the first claim of the lemma.
\end{proof}

\begin{lemma} \label{lemma, multisets of Hankel matrices}
Let $H$ be an $n \times n$ Hankel matrix with \rhopi partition equal to\break $(p_1 ' , p_1 '' , p_2 , \ldots , p_t )$, and so $\reduced (H)$ is of the form
\begin{align*}
\begin{pNiceMatrix}
H_t & \mathbf{0} & \Cdots &  &  & \mathbf{0} \\
\mathbf{0} & H_{t-1} & \Ddots &  &  & \Vdots \\
\Vdots & \Ddots & \Ddots &  &  &  \\
 &  &  & H_2 & \mathbf{0} & \mathbf{0} \\
 &  &  & \mathbf{0} & H_1 '' & \mathbf{0} \\
\mathbf{0} & \Cdots &  & \mathbf{0} & \mathbf{0} & H_1 '
\end{pNiceMatrix} .
\end{align*}
Of the integers $p_2 , \ldots , p_t$, suppose there are exactly $s$ of them that are odd, and denote these by $p_{i_1} , \ldots , p_{i_s}$. Let $\lambda_{i_1} , \ldots , \lambda_{i_s}$ be the main skew diagonal entries of $H_{i_1} , \ldots , H_{i_s}$. Also, if $p_1' \geq 1$, then let $\lambda_1 \neq 0$ be the main skew-diagonal entry of $H_1'$. Then,
\begin{align*}
\multisetmatrix_{\mathcal{M}} (H)
= \frac{(p_2 + \ldots + p_t ) - s}{2} T_q
	+ S_q ( \lambda_{i_1}) + \ldots + S_q ( \lambda_{i_s})
	+ G ,
\end{align*}
where
\begin{align*}
G
= \begin{cases}
Z_q (q^{p_1 ''}) + F_q + \frac{p_1' -2}{2} T_q &\text{if $p_1 ' \geq 2$ is even,} \\
Z_q (q^{p_1 ''})  + F_q + \frac{p_1' -3}{2} T_q + S_q (\lambda_1 ) &\text{if $p_1 ' \geq 3$ is odd,} \\
Z_q (q^{p_1 ''})  + [ \lambda_1 ] &\text{if $p_1 ' = 1$,} \\
Z_q (q^{p_1 '' -1})  + [0] = Z_q (q^{p_1 '' -1}) &\text{if $p_1 ' = 0$.}
\end{cases}
\end{align*}
\end{lemma}

\begin{proof}
By Lemma \ref{lemma, multiset H equals multiset reduced H}, we have
\begin{align*}
\multisetmatrix_{\mathcal{M}} (H)
= \multisetmatrix_{\mathcal{M}} \big( \reduced (H) \big) .
\end{align*}
By the block diagonal nature of $\reduced (H)$, we have
\begin{align*}
\multisetmatrix_{\mathcal{M}} \big( \reduced (H) \big) 
= \begin{cases}
\multisetmatrix_{\mathcal{M}} (H_1 ') + \multisetmatrix (H_1 '')+ \sum_{i=2}^{t} \multisetmatrix (H_i ) &\text{if $p_1 ' \geq 1$,} \\
\multisetmatrix_{\mathcal{M}} (H_1 '') + \sum_{i=2}^{t} \multisetmatrix (H_i ) &\text{if $p_1 ' = 0$.} 
\end{cases}
\end{align*}
As $H_1 ' , H_2 , \ldots , H_t$ are non-strict lower skew-triangular Hankel matrices, we can apply Lemma \ref{lemma, multiset of nonstrict lower skew triang matr} to $\multisetmatrix (H_1 ' ) , \multisetmatrix (H_2 ) , \ldots , \multisetmatrix (H_t )$. The result follows from this and the fact $H_1 ''$ being a $p_1 '' \times p_1 ''$ zero matrix implies
\begin{align*}
\multisetmatrix (H_1 '') 
= Z_q (q^{p_1 '' })
\end{align*}
and
\begin{align*}
\multisetmatrix_{\mathcal{M}} (H_1 '') 
= Z_q (q^{p_1 '' -1}) .
\end{align*}
\end{proof}

\section{Proof of Theorem \ref{main theorem, variance sums of squares}} \label{section, main theorems proofs}

For integers $k,h$ satisfying $0 \leq h \leq k$, we will need to define 
\begin{align*}
\mathscr{L}_k^h
:= \{ (\alpha_0 , \alpha_1 , \ldots , \alpha_{k}) \in \mathbb{F}_q^{k+1} : \alpha_0 , \ldots , \alpha_{h-1} = 0 \} . \\
\end{align*}

\begin{proof}[Proof of Theorem \ref{main theorem, variance sums of squares}]
As we have established in Section \ref{section, introduction}, we have that
\begin{align} \label{statement, main theorem proof, variance first simplification}
\variance{2n,h}{S_{\gamma ;m}}
= \frac{1}{q^{2n}} \sum_{A \in \mathcal{M}_{2n}} \Big( \sum_{B \in I(A;h)} S_{\gamma ;m} (B) \Big)^2 
	- q^{2(m+h-n)} .
\end{align}
Thus, we will focus on $\sum_{A \in \mathcal{M}_{2n}} \Big( \sum_{B \in I(A;h)} S_{\gamma ;m} (B) \Big)^2$. We have
\begin{align*}
\sum_{A \in \mathcal{M}_{2n}} \Big( \sum_{B \in I(A;h)} S_{\gamma ;m} (B) \Big)^2 
= \sum_{A \in \mathcal{M}_{2n} }
	\bigg( \sum_{\substack{B \in \mathcal{M}_{2n} \\ \degree (B-A) < h }} 
		\sum_{\substack{ E \in \mathcal{M}_n \\ F \in \mathcal{M}_m \\ E^2 + \gamma F^2 = B }} 1 \bigg)^2 
= \sum_{A \in \mathcal{A}_{\leq 2n} }
	\bigg( \sum_{\substack{B \in \mathcal{A}_{\leq 2n} \\ \degree (B-A) < h }} 
		\sum_{\substack{ E \in \mathcal{M}_n \\ F \in \mathcal{M}_m \\ E^2 + \gamma F^2 = B }} 1 \bigg)^2 .
\end{align*}
For the last equality, we can replace $A \in \mathcal{M}_{2n}$ and $B \in \mathcal{M}_{2n}$ with $A \in \mathcal{A}_{\leq 2n}$ and $B \in \mathcal{A}_{\leq 2n}$ because the condition $E^2 + \gamma F^2 = B$ forces $B$ and $A$ to be in $\mathcal{M}_{2n}$. For the polynomials $A,B,E,F$ we denote their $i$-th coefficient by $a_i , b_i , e_i , f_i$, respectively. Generally, for $C \in \mathcal{A}$ let us denote its $i$-th coefficient by $\{ C \}_i$; this is more convenient when $C$ is expressed as a product. Also, in what follows, $\mathbf{a}, \mathbf{b}, \mathbf{e}, \mathbf{f}$ are the vectors whose $i$-th term is equal to the $i$-th coefficient of $A,B,E,F$, respectively. We have
\begin{align*}
&\sum_{A \in \mathcal{M}_{2n}} \Big( \sum_{B \in I(A;h)} S_{\gamma ;m} (B) \Big)^2 \\
= &\sum_{A \in \mathcal{A}_{\leq 2n} }
	\bigg( \sum_{\substack{B \in \mathcal{A}_{\leq 2n} \\ \degree (B-A) < h }} 
		\sum_{\substack{ E \in \mathcal{M}_n \\ F \in \mathcal{M}_m }}
		\prod_{i=0}^{2n} \mathbbm{1} ( \{ E^2 \}_i + \gamma \{ F^2 \}_i = \{ B \}_i ) \bigg)^2 \\
= &\frac{1}{q^{4n+2}} \sum_{A \in \mathcal{A}_{\leq 2n} }
	\bigg( \sum_{\substack{B \in \mathcal{A}_{\leq 2n} \\ \degree (B-A) < h }} 
		\sum_{\substack{ E \in \mathcal{M}_n \\ F \in \mathcal{M}_m }}
		\prod_{i=0}^{2n} 
		\sum_{\alpha_i \in \mathbb{F}_q } 
		\psi \Big( \alpha_i \big( \{ E^2 \}_i + \gamma \{ F^2 \}_i - \{ B \}_i \big) \Big)
		\bigg)^2 .
\end{align*}
For the last equality, we have used (\ref{statement, intro, exp identity sum}) with $\alpha = \alpha_i$ and $b = \{ E^2 \}_i + \gamma \{ F^2 \}_i - \{ B \}_i$. Writing $\mathbfe{\alpha} = (\alpha_0 , \alpha_1 , \ldots , \alpha_{2n})$, we have
\begin{align*}
&\sum_{A \in \mathcal{M}_{2n}} \Big( \sum_{B \in I(A;h)} S_{\gamma ;m} (B) \Big)^2 \\
= &\frac{1}{q^{4n-h+2}} \sum_{a_h , \ldots , a_{2n} \in \mathbb{F}_q }
	\Bigg( \sum_{b_0 , \ldots , b_{h-1} \in \mathbb{F}_q} 
		\sum_{\substack{ E \in \mathcal{M}_n \\ F \in \mathcal{M}_m }}
		\sum_{\mathbfe{\alpha} \in \mathbb{F}_q^{2n+1} } 
		\prod_{i=0}^{h-1} \bigg(
		\psi \Big( \alpha_i \big( \{ E^2 \}_i + \gamma \{ F^2 \}_i - b_i \big) \Big) \bigg) \\
		&\hspace{20em} \times \prod_{i=h}^{2n} \bigg(
		\psi \Big( \alpha_i \big( \{ E^2 \}_i + \gamma \{ F^2 \}_i - a_i \big) \Big) \bigg)
		\Bigg)^2 \\
= &\frac{1}{q^{4n-h+2}} \sum_{a_h , \ldots , a_{2n} \in \mathbb{F}_q }
	\Bigg( \sum_{b_0 , \ldots , b_{h-1} \in \mathbb{F}_q} 
		\sum_{\substack{ E \in \mathcal{M}_n \\ F \in \mathcal{M}_m }}
		\sum_{\mathbfe{\alpha} \in \mathbb{F}_q^{2n+1} } 
		\prod_{i=0}^{h-1} \bigg(
		\psi \Big( \alpha_i \big( \{ E^2 \}_i + \gamma \{ F^2 \}_i - b_i \big) \Big) \bigg) \\
		&\hspace{20em} \times \prod_{i=h}^{2n} \bigg(
		\psi \Big( \alpha_i \big( \{ E^2 \}_i + \gamma \{ F^2 \}_i - a_i \big) \Big) \bigg)
		\Bigg) \\
	&\hspace{8em} \times \Bigg( \sum_{b_0' , \ldots , b_{h-1}' \in \mathbb{F}_q} 
		\sum_{\substack{ E' \in \mathcal{M}_n \\ F' \in \mathcal{M}_m }}
		\sum_{\mathbfe{\beta} \in \mathbb{F}_q^{2n+1} } 
		\prod_{i=0}^{h-1} \bigg(
		\psi \Big( \beta_i \big( \{ E'^2 \}_i + \gamma \{ F'^2 \}_i - b_i' \big) \Big) \bigg) \\
		&\hspace{20em} \times \prod_{i=h}^{2n} \bigg(
		\psi \Big( \beta_i \big( \{ E'^2 \}_i + \gamma \{ F'^2 \}_i - a_i \big) \Big) \bigg)
		\Bigg) .
\end{align*}
For $i = 0 , \ldots , h-1$, let us consider the sum over $b_i$. We have
\begin{align*}
\frac{1}{q} \sum_{b_i \in \mathbb{F}_q} \psi (- \alpha_i b_i) .
\end{align*}
By (\ref{statement, intro, exp identity sum}), this is non-zero (specifically, it takes the value $1$) if and only if $\alpha_i = 0$. We can apply the same reasoning to the sums over $b_i'$ and we deduce that $\beta_i$ = 0 for $i = 0 , \ldots , h-1$. Let us also consider the sum over $a_i$, for $i = h , \ldots , 2n$. We have
\begin{align*}
\frac{1}{q} \sum_{a_i \in \mathbb{F}_q} \psi \big( -( \alpha_i + \beta_i ) a_i \big) .
\end{align*}
Again, we apply (\ref{statement, intro, exp identity sum}) to see that this takes the value $1$ if $\alpha_i + \beta_i = 0$, and takes the value $0$ otherwise. Thus, we can see that $\mathbfe{\alpha} = (0 , \ldots , \alpha_h , \ldots , \alpha_{2n}) \in \mathscr{L}_{2n}^h$ and $\mathbfe{\beta} = (0 , \ldots , - \alpha_h , \ldots , - \alpha_{2n})$. Hence,
\begin{align}
\begin{split}\label{statement, main theorem with m cases, first Hankel expression}
&\sum_{A \in \mathcal{M}_{2n}} \Big( \sum_{B \in I(A;h)} S_{\gamma ;m} (B) \Big)^2 \\
= &\frac{1}{q^{2n-2h+1}} \sum_{ \mathbfe{\alpha} \in \mathscr{L}_{2n}^h }
	\Bigg( \sum_{\substack{ E \in \mathcal{M}_n \\ F \in \mathcal{M}_m }}
		\prod_{i=0}^{2n}
		\psi \Big( \alpha_i \big( \{ E^2 \}_i + \gamma \{ F^2 \}_i \big) \Big) \Bigg)
	\Bigg( \sum_{\substack{ E \in \mathcal{M}_n \\ F \in \mathcal{M}_m }}
		\prod_{i=0}^{2n}
		\psi \Big( - \alpha_i \big( \{ E^2 \}_i + \gamma \{ F^2 \}_i \big) \Big) \Bigg) \\
= &\frac{1}{q^{2n-2h+1}} \sum_{ \mathbfe{\alpha} \in \mathscr{L}_{2n}^h }
	\Bigg( \sum_{\substack{ \mathbf{e} \in \mathbb{F}_q^{n} \times \{ 1 \} \\ \mathbf{f} \in \mathbb{F}_q^{m} \times \{ 1 \} }}
		\psi \bigg( \mathbf{e}^T H_{n+1 , n+1} (\mathbfe{\alpha}) \mathbf{e}
			+ \mathbf{f}^T \gamma H_{m+1 , m+1} (\mathbfe{\alpha}) \mathbf{f}
		\bigg) \Bigg) \\
	& \hspace{6em} \times \Bigg( \sum_{\substack{ \mathbf{e} \in \mathbb{F}_q^{n} \times \{ 1 \} \\ \mathbf{f} \in \mathbb{F}_q^{m} \times \{ 1 \} }}
		\psi \bigg( \mathbf{e}^T H_{n+1 , n+1} (- \mathbfe{\alpha}) \mathbf{e}
			+ \mathbf{f}^T \gamma H_{m+1 , m+1} (- \mathbfe{\alpha}) \mathbf{f}
		\bigg) \Bigg) .
\end{split}
\end{align}
Now, Lemma \ref{lemma, main theorem alpha contributions} below tells us that a non-zero contribution to (\ref{statement, main theorem with m cases, first Hankel expression}) occurs only when $\pi_s \Big( H_{n+1 , n+1} (\mathbfe{\alpha}) \Big) \in \{ 0 , 1 \}$ and $\pi_s \Big( \gamma H_{m+1 , m+1} (\mathbfe{\alpha}) \Big) \in \{ 0 , 1 \}$, and the value of the contribution depends on $\rho_s \Big( H_{n+1 , n+1} (\mathbfe{\alpha}) \Big)$ and $\rho_s \Big( \gamma H_{m+1 , m+1} (\mathbfe{\alpha}) \Big)$. Note that multiplying $H_{m+1 , m+1} (\mathbfe{\alpha})$ by $\gamma \in \mathbb{F}_q^*$ does not change its strict \rhopi-characteristic. From this and the second claim of Lemma \ref{lemma, main theorem alpha contributions}, we can, without loss of generality, take $\gamma =1$. Thus, we must determine the value of 
\begin{align*}
N_{n,m,h} (\rho_2 , \rho_1)
:= \bigg\lvert \bigg\{ \mathbfe{\alpha} \in \mathscr{L}_{2n}^h : \substack{ 
	\rho_s \Big( H_{m+1 , m+1} (\mathbfe{\alpha}) \Big) = \rho_2 , 
		\hspace{1em} \pi_s \Big( H_{m+1 , m+1} (\mathbfe{\alpha}) \Big) \in \{ 0 , 1 \} \\
	\rho_s \Big( H_{n+1 , n+1} (\mathbfe{\alpha}) \Big) = \rho_1 , 
		\hspace{1em} \pi_s \Big( H_{n+1 , n+1} (\mathbfe{\alpha}) \Big) \in \{ 0 , 1 \} } \bigg\} \bigg\rvert ,
\end{align*}
and what the permissible ranges for $\rho_2 , \rho_1$ are. This is determined in Lemma \ref{lemma, main theorem N_m,n,h (rho_1 , rho_2) evaluation}. The ranges for $\rho_2 , \rho_1$ are dependent on $n,m,h$, and so we must consider several cases and subcases separately. Before we do this, let us use Lemmas \ref{lemma, main theorem alpha contributions} and \ref{lemma, main theorem N_m,n,h (rho_1 , rho_2) evaluation} to rewrite (\ref{statement, main theorem with m cases, first Hankel expression}) as
\begin{align*}
\sum_{A \in \mathcal{M}_{2n}} \Big( \sum_{B \in I(A;h)} S_{\gamma ;m} (B) \Big)^2 
= q^{2m+2h-1}
	\sum_{\rho_2 , \rho_1}
	N_{n,m,h} (\rho_2 , \rho_1)
	q^{-\rho_1 - \rho_2} ,
\end{align*}
where the sum on the right side is over the permissible values of $\rho_2 , \rho_1$. \\

Now, applying Lemma \ref{lemma, main theorem N_m,n,h (rho_1 , rho_2) evaluation}, we have the following cases. \\

\textbf{Case 1:} $m+1 \leq n \leq 2m$. \\

\underline{Subcase 1.1:} $m \leq h \leq 2n$. \\

\begin{align*}
\sum_{A \in \mathcal{M}_{2n}} \Big( \sum_{B \in I(A;h)} S_{\gamma ;m} (B) \Big)^2 
= q^{2m+2h-1}
	\sum_{\rho_2 , \rho_1 = 0} q 
= q^{2m+2h} . 
\end{align*}

\underline{Subcase 1.2:} $2m-n \leq h \leq m-1$. \\

\begin{align*}
&\sum_{A \in \mathcal{M}_{2n}} \Big( \sum_{B \in I(A;h)} S_{\gamma ;m} (B) \Big)^2 \\
= &q^{2m+2h-1} \Bigg[
	\bigg( \sum_{\rho_2 , \rho_1 = 0} q \bigg)
	+ \bigg( \sum_{h+1 \leq \rho_2 = \rho_1 \leq m} (q-1) q^{-h} \bigg)
	+ \bigg( \sum_{\substack{h+1 \leq \rho_2 \leq m \\ 2m+1-\rho_2 \leq \rho_1 \leq n}} (q-1)^2 q^{\rho_1 + \rho_2 -2m-h-1} \bigg) \Bigg] \\
= &q^{2m+2h-1} \bigg[
	q
	+ (m-h) (q-1) q^{-h}
	+ \Big( q^{n-2m-h} (q^{m+1} - q^{h+1} ) - (m-h) (q-1) q^{-h} \Big) \bigg) \\
= &q^{2m+2h} + q^{n+m+h} + q^{n+2h} . \\
\end{align*}

\underline{Subcase 1.3:} $0 \leq h \leq 2m-n-1$. (This does not apply when $n=2m$). \\

\begin{align*}
&\sum_{A \in \mathcal{M}_{2n}} \Big( \sum_{B \in I(A;h)} S_{\gamma ;m} (B) \Big)^2 \\
= &q^{2m+2h-1} \Bigg[
	\bigg( \sum_{\rho_2 , \rho_1 = 0} q \bigg)
	+ \bigg( \sum_{h+1 \leq \rho_2 = \rho_1 \leq m} (q-1) q^{-h} \bigg)
	+ \bigg( \sum_{\substack{2m+1-n \leq \rho_2 \leq m \\ 2m+1-\rho_2 \leq \rho_1 \leq n}} (q-1)^2 q^{\rho_1 + \rho_2 -2m-h-1} \bigg) \Bigg] \\
= &q^{2m+2h-1} \bigg[
	q
	+ (m-h) (q-1) q^{-h}
	+ \Big( q^{n-m-h+1} - q^{-h+1} - (n-m) (q-1) q^{-h} \Big) \bigg) \\
= &q^{2m+2h} + (2m-n-h) (q-1) q^{2m+h-1} + q^{n+m+h} + q^{2m+h} . \\
\end{align*}

\textbf{Case 2:} $n \geq 2m+1$. \\

\underline{Subcase 2.1:} $n \leq h \leq 2n$. \\

\begin{align*}
\sum_{A \in \mathcal{M}_{2n}} \Big( \sum_{B \in I(A;h)} S_{\gamma ;m} (B) \Big)^2 
= q^{2m+2h-1}
	\sum_{\rho_2 , \rho_1 = 0} q 
= q^{2m+2h} . 
\end{align*}

\underline{Subcase 2.2:} $2m \leq h \leq n-1$. \\

\begin{align*}
\sum_{A \in \mathcal{M}_{2n}} \Big( \sum_{B \in I(A;h)} S_{\gamma ;m} (B) \Big)^2 
= q^{2m+2h-1}
	\bigg( \sum_{\rho_2 , \rho_1 = 0} q 
		+\sum_{\substack{\rho_2 = 0 \\ h+1 \leq \rho_1 \leq n}} (q-1) q^{\rho_1 - h} \bigg)
= q^{n+2m+h} . 
\end{align*}

\underline{Subcase 2.3:} $m \leq h \leq 2m-1$. \\

\begin{align*}
\sum_{A \in \mathcal{M}_{2n}} \Big( \sum_{B \in I(A;h)} S_{\gamma ;m} (B) \Big)^2 
= q^{2m+2h-1}
	\bigg( \sum_{\rho_2 , \rho_1 = 0} q 
		+\sum_{\substack{\rho_2 = 0 \\ 2m+1 \leq \rho_1 \leq n}} (q-1) q^{\rho_1 - 2m} \bigg)
= q^{n+2h} . 
\end{align*}

\underline{Subcase 2.4:} $0 \leq h \leq m-1$. \\

\begin{align*}
&\sum_{A \in \mathcal{M}_{2n}} \Big( \sum_{B \in I(A;h)} S_{\gamma ;m} (B) \Big)^2 \\
= &q^{2m+2h-1}
	\bigg( \sum_{\rho_2 , \rho_1 = 0} q 
		+\sum_{\substack{\rho_2 = 0 \\ 2m+1 \leq \rho_1 \leq n}} (q-1) q^{\rho_1 - 2m} \\
		&\hspace{6em} +\sum_{h+1 \leq \rho_2 = \rho_1 \leq m} (q-1) q^{- h} 
		+\sum_{\substack{h+1 \leq \rho_2 \leq m \\ 2m+1 - \rho_2 \leq \rho_1 \leq n}} (q-1)^2 q^{\rho_1 + \rho_2 -2m- h-1} \bigg) \\
= &q^{2m+2h-1}
	\bigg( (q)
		+ (q^{n-2m+1} - q)
		+(m-h) (q-1) q^{-h}
		+ \Big( q^{n-m-h+1} - q^{n-2m+1} - (m-h) (q-1) q^{-h} \Big)\bigg) \\
= &q^{n+m+h} .\\
\end{align*}

\textbf{Special Case:} When $m=0$ only Case 2, Subcases 2.1 and 2.2 apply. \\

The proof follows by substituting each of the subcases above into (\ref{statement, main theorem proof, variance first simplification}) .
\end{proof}

\vspace{1em}

\begin{lemma} \label{lemma, main theorem alpha contributions}
Let $\mathbfe{\alpha} \in \mathbb{F}_q^{2n+1}$, and let $(p_1 ' , p_1 '' , p_2 , \ldots , p_t )$ be the \rhopi-partition of $H_{n+1 , n+1} (\mathbfe{\alpha})$. If $p_1 ' \geq 2$, then
\begin{align*}
\sum_{ \mathbf{e} \in \mathbb{F}_q^{n} \times \{ 1 \} }
	\psi \bigg( \mathbf{e}^T H_{n+1 , n+1} (\mathbfe{\alpha}) \mathbf{e} \bigg) 
= 0 ,
\end{align*}
from which we immediately deduce
\begin{align*}
\Bigg( \sum_{ \mathbf{e} \in \mathbb{F}_q^{n} \times \{ 1 \} }
	\psi \bigg( \mathbf{e}^T H_{n+1 , n+1} (\mathbfe{\alpha}) \mathbf{e} \bigg) \Bigg)
\Bigg( \sum_{ \mathbf{e} \in \mathbb{F}_q^{n} \times \{ 1 \} }
	\psi \bigg( \mathbf{e}^T H_{n+1 , n+1} (- \mathbfe{\alpha}) \mathbf{e} \bigg) \Bigg)
= 0 .
\end{align*}

If $0 \leq p_1 ' \leq 1$, then 
\begin{align*}
\Bigg( \sum_{ \mathbf{e} \in \mathbb{F}_q^{n} \times \{ 1 \} }
	\psi \bigg( \mathbf{e}^T H_{n+1 , n+1} (\mathbfe{\alpha}) \mathbf{e} \bigg) \Bigg)
\Bigg( \sum_{ \mathbf{e} \in \mathbb{F}_q^{n} \times \{ 1 \} }
	\psi \bigg( \mathbf{e}^T H_{n+1 , n+1} (- \mathbfe{\alpha}) \mathbf{e} \bigg) \Bigg)
= q^{2n - \rho_1} ,
\end{align*}
where $(\rho_1 , \pi_1)$ is the strict \rhopi-characteristic of $H_{n+1 , n+1} (\mathbfe{\alpha})$. (Recall that $\rho_1 = p_2 + \ldots + p_t$ and $\pi_1 = p_1 '$).
\end{lemma}

\begin{proof}
We begin with the first claim, where $p_1 ' \geq 2$. We have
\begin{align*}
\sum_{ \mathbf{e} \in \mathbb{F}_q^{n} \times \{ 1 \} }
	\psi \bigg( \mathbf{e}^T H_{n+1 , n+1} (\mathbfe{\alpha}) \mathbf{e} \bigg) 
= \sum_{ a \in \multisetmatrix_{\mathcal{M}} \big( H_{n+1 , n+1} (\mathbfe{\alpha}) \big) }
	\psi ( a ) .
\end{align*}
Now, Lemma \ref{lemma, multisets of Hankel matrices} tells us that 
\begin{align*}
\multisetmatrix_{\mathcal{M}} \big( H_{n+1 , n+1} (\mathbfe{\alpha}) \big)
= J + F_q ,
\end{align*}
where $J$ is a multiset that the Lemma gives explicitly (but we do not require it explicitly for this proof), and we recall that $F_q$ is defined to be the multiset where each element of $\mathbb{F}_q$ appears exactly once. Thus, we have
\begin{align*}
\sum_{ \mathbf{e} \in \mathbb{F}_q^{n} \times \{ 1 \} }
	\psi \bigg( \mathbf{e}^T H_{n+1 , n+1} (\mathbfe{\alpha}) \mathbf{e} \bigg) 
= \sum_{ a \in J+F_q } \psi ( a ) 
= \bigg( \sum_{ a \in J } \psi ( a ) \bigg) \bigg( \sum_{ b \in F_q } \psi ( b )  \bigg)
= 0 ,
\end{align*}
where the last equality uses the fact that $\sum_{ b \in F_q } \psi ( b ) = \sum_{ b \in \mathbb{F}_q } \psi ( b ) = 0$. \\

Now we consider the second claim, where $0 \leq p_1 ' \leq 1$. As in Lemma \ref{lemma, multisets of Hankel matrices}, of the integers $p_2 , \ldots , p_t$, suppose there are exactly $s$ of them that are odd, and denote these by $p_{i_1} , \ldots , p_{i_s}$. Let $\lambda_{i_1} , \ldots , \lambda_{i_s}$ be the main skew-diagonal entries of the corresponding blocks in $\reduced \big( H_{n+1 , n+1} (\mathbfe{\alpha}) \big)$. Also, if $p_1 ' = 1$, then let $\lambda$ be the main skew-diagonal entry of $H_1 '$ (which is the only entry of $H_1 '$ as it is a $1 \times 1$ matrix). Otherwise, if $p_1 ' = 0$, then let $\lambda = 0$. Lemma \ref{lemma, multisets of Hankel matrices} gives us
\begin{align*}
&\sum_{ \mathbf{e} \in \mathbb{F}_q^{n} \times \{ 1 \} }
	\psi \bigg( \mathbf{e}^T H_{n+1 , n+1} (\mathbfe{\alpha}) \mathbf{e} \bigg) \\
= &\sum_{ a \in \multisetmatrix_{\mathcal{M}} \big( H_{n+1 , n+1} (\mathbfe{\alpha}) \big) }
	\psi ( a ) \\
= &\bigg( \sum_{a \in \frac{(p_2 + \ldots + p_t ) - s}{2} T_q } \psi (a) \bigg)
	\bigg( \sum_{b \in S_q ( \lambda_{i_1}) + \ldots + S_q ( \lambda_{i_s}) } \psi (b) \bigg)
	\bigg( \sum_{c \in Z_q (q^{p_1 ' + p_1 '' -1})} \psi (c) \bigg)
	\psi (\lambda ) .
\end{align*}
Similarly,
\begin{align*}
&\sum_{ \mathbf{e} \in \mathbb{F}_q^{n} \times \{ 1 \} }
	\psi \bigg( \mathbf{e}^T H_{n+1 , n+1} (- \mathbfe{\alpha}) \mathbf{e} \bigg) \\
= &\bigg( \sum_{a \in \frac{(p_2 + \ldots + p_t ) - s}{2} T_q } \psi (a) \bigg)
	\bigg( \sum_{b \in S_q ( - \lambda_{i_1}) + \ldots + S_q ( - \lambda_{i_s}) } \psi (b) \bigg)
	\bigg( \sum_{c \in Z_q (q^{p_1 ' + p_1 '' -1})} \psi (c) \bigg)
	\psi (- \lambda ) .
\end{align*}
Now, for any $\mu \in \mathbb{F}_q^*$, it is not difficult to see that
\begin{align*}
S_q (\mu ) + S_q (- \mu )
= T_q .
\end{align*}
Thus, 
\begin{align*}
&\bigg( \sum_{b \in S_q ( \lambda_{i_1}) + \ldots + S_q ( \lambda_{i_s}) } \psi (b) \bigg)
	\bigg( \sum_{b \in S_q ( - \lambda_{i_1}) + \ldots + S_q ( - \lambda_{i_s}) } \psi (b) \bigg) \\
= &\bigg( \sum_{b_1 \in S_q ( \lambda_{i_1}) + S_q ( - \lambda_{i_1}) } \psi (b_1 ) \bigg)
	\ldots
	\bigg( \sum_{b_s \in S_q ( \lambda_{i_s}) + S_q ( - \lambda_{i_s}) } \psi (b_s ) \bigg) \\
= &\bigg( \sum_{b \in T_q } \psi (b) \bigg)^s 
= \bigg( \sum_{b \in sT_q } \psi (b) \bigg) .
\end{align*}
Also,
\begin{align*}
\sum_{c \in Z_q (q^{p_1 ' + p_1 '' -1})} \psi (c) 
= q^{p_1 ' + p_1 '' -1}
= q^{n - \rho_1} ,
\end{align*}
where the last line uses the fact that $p_1' + p_1'' + p_2 + \ldots + p_t = n+1$ and $p_2 + \ldots + p_t = \rho_1$. So, we can deduce that
\begin{align*}
&\Bigg( \sum_{ \mathbf{e} \in \mathbb{F}_q^{n} \times \{ 1 \} }
	\psi \bigg( \mathbf{e}^T H_{n+1 , n+1} (\mathbfe{\alpha}) \mathbf{e} \bigg) \Bigg)
\Bigg( \sum_{ \mathbf{e} \in \mathbb{F}_q^{n} \times \{ 1 \} }
	\psi \bigg( \mathbf{e}^T H_{n+1 , n+1} (- \mathbfe{\alpha}) \mathbf{e} \bigg) \Bigg) \\
= &q^{2n - 2 \rho_1}
	\bigg( \sum_{a \in \frac{(p_2 + \ldots + p_t ) - s}{2} T_q } \psi (a) \bigg)^2
	\bigg( \sum_{b \in sT_q } \psi (b) \bigg) \\
= &q^{2n - 2 \rho_1}
	\sum_{a \in \rho_1 T_q } \psi (a) 
= q^{2n - \rho_1} ,
\end{align*}
where the second-to-last equality uses the fact that $\rho_1 = p_2 + \ldots + p_t$, and the last equality uses the fact that $\sum_{a \in T_q } \psi (a) = q$.
\end{proof}

\begin{lemma} \label{lemma, main theorem N_m,n,h (rho_1 , rho_2) evaluation}
Define 
\begin{align} \label{statement, lemma, main theorem N_m,n,h (rho_1 , rho_2) evaluation, N_n,m,h def}
N_{n,m,h} (\rho_2 , \rho_1)
:= \bigg\lvert \bigg\{ \mathbfe{\alpha} \in \mathscr{L}_{2n}^h : \substack{ 
	\rho_s \Big( H_{m+1 , m+1} (\mathbfe{\alpha}) \Big) = \rho_2 , 
		\hspace{1em} \pi_s \Big( H_{m+1 , m+1} (\mathbfe{\alpha}) \Big) \in \{ 0 , 1 \} \\
	\rho_s \Big( H_{n+1 , n+1} (\mathbfe{\alpha}) \Big) = \rho_1 , 
		\hspace{1em} \pi_s \Big( H_{n+1 , n+1} (\mathbfe{\alpha}) \Big) \in \{ 0 , 1 \} } \bigg\} \bigg\rvert .
\end{align}
In what follows, we have several cases and subcases that cover all possible values for $n,m,h$. In each subcase, we give the possible values of $\rho_2 , \rho_1$, and the associated value of $N_{n,m,h} (\rho_2 , \rho_1)$. For any values of $\rho_2 , \rho_1$ that are not explicitly given, it should be interpreted that \break $N_{n,m,h} (\rho_2 , \rho_1) = 0$ for those values. It may also be helpful for the reader to note that if $h' \geq h$, then $\mathscr{L}_{2n}^{h'} \subseteq \mathscr{L}_{2n}^h$, and so the results for the higher ranges of $h$ will also appear in the lower ranges of $h$, although the latter will have additional results. \\

\textbf{Case 1:} $m+1 \leq n \leq 2m$. \\

\underline{Subcase 1.1:} $m \leq h \leq 2n$. 

\begin{align*}
\begin{tabular}{ |c|c|c| }
\hline
\textbf{Range of $\rho_2$} & \textbf{Range of $\rho_1$} & \textbf{Value of $N_{n,m,h} (\rho_2 , \rho_1)$} \\
\hline
$\rho_2 = 0$ & $\rho_1 = 0$ & $q$ \\
\hline
\end{tabular} \\
\end{align*}

\underline{Subcase 1.2:} $2m-n \leq h \leq m-1$. \\

\begin{align*}
\begin{tabular}{ |c|c|c| }
\hline
\textbf{Range of $\rho_2$} & \textbf{Range of $\rho_1$} & \textbf{Value of $N_{n,m,h} (\rho_2 , \rho_1)$} \\
\hline
$\rho_2 = 0$ & $\rho_1 = 0$ & $q$ \\
\hline
$h+1 \leq \rho_2 \leq m$ & $\rho_1 = \rho_2$ & $(q-1) q^{2 \rho_2 -h}$ \\
\hline
$h+1 \leq \rho_2 \leq m$ & $2m+1- \rho_2 \leq \rho_1 \leq n$ & $(q-1)^2 q^{2 \rho_1 + 2 \rho_2 -2m -h -1}$ \\
\hline
\end{tabular} \\
\end{align*}

\underline{Subcase 1.3:} $0 \leq h \leq 2m-n-1$. (This does not apply when $n=2m$). \\

\begin{align*}
\begin{tabular}{ |c|c|c| }
\hline
\textbf{Range of $\rho_2$} & \textbf{Range of $\rho_1$} & \textbf{Value of $N_{n,m,h} (\rho_2 , \rho_1)$} \\
\hline
$\rho_2 = 0$ & $\rho_1 = 0$ & $q$ \\
\hline
$h+1 \leq \rho_2 \leq m$ & $\rho_1 = \rho_2$ & $(q-1) q^{2 \rho_2 -h}$ \\
\hline
$2m+1-n \leq \rho_2 \leq m$ & $2m+1- \rho_2 \leq \rho_1 \leq n$ & $(q-1)^2 q^{2 \rho_1 + 2 \rho_2 -2m -h -1}$ \\
\hline
\end{tabular} \\
\end{align*}

\textbf{Case 2:} $n \geq 2m+1$. \\

\underline{Subcase 2.1:} $n \leq h \leq 2n$. 

\begin{align*}
\begin{tabular}{ |c|c|c| }
\hline
\textbf{Range of $\rho_2$} & \textbf{Range of $\rho_1$} & \textbf{Value of $N_{n,m,h} (\rho_2 , \rho_1)$} \\
\hline
$\rho_2 = 0$ & $\rho_1 = 0$ & $q$ \\
\hline
\end{tabular} \\
\end{align*}

\underline{Subcase 2.2:} $2m \leq h \leq n-1$. 

\begin{align*}
\begin{tabular}{ |c|c|c| }
\hline
\textbf{Range of $\rho_2$} & \textbf{Range of $\rho_1$} & \textbf{Value of $N_{n,m,h} (\rho_2 , \rho_1)$} \\
\hline
$\rho_2 = 0$ & $\rho_1 = 0$ & $q$ \\
\hline
$\rho_2 = 0$ & $h+1 \leq \rho_1 \leq n$ & $(q-1) q^{2 \rho_1 -h}$ \\
\hline
\end{tabular} \\
\end{align*}

\underline{Subcase 2.3:} $m \leq h \leq 2m-1$. 

\begin{align*}
\begin{tabular}{ |c|c|c| }
\hline
\textbf{Range of $\rho_2$} & \textbf{Range of $\rho_1$} & \textbf{Value of $N_{n,m,h} (\rho_2 , \rho_1)$} \\
\hline
$\rho_2 = 0$ & $\rho_1 = 0$ & $q$ \\
\hline
$\rho_2 = 0$ & $2m+1 \leq \rho_1 \leq n$ & $(q-1) q^{2 \rho_1 -2m}$ \\
\hline
\end{tabular} \\
\end{align*}

\underline{Subcase 2.4:} $0 \leq h \leq m-1$. 

\begin{align*}
\begin{tabular}{ |c|c|c| }
\hline
\textbf{Range of $\rho_2$} & \textbf{Range of $\rho_1$} & \textbf{Value of $N_{n,m,h} (\rho_2 , \rho_1)$} \\
\hline
$\rho_2 = 0$ & $\rho_1 = 0$ & $q$ \\
\hline
$\rho_2 = 0$ & $2m+1 \leq \rho_1 \leq n$ & $(q-1) q^{2 \rho_1 -2m}$ \\
\hline
$h+1 \leq \rho_2 \leq m$ & $\rho_1 = \rho_2$ & $(q-1) q^{2 \rho_1 -h}$ \\
\hline
$h+1 \leq \rho_2 \leq m$ & $2m+1 -\rho_2 \leq \rho_1 \leq n$ & $(q-1)^2 q^{2 \rho_1 +2 \rho_2 -2m -h-1}$ \\
\hline
\end{tabular} \\
\end{align*}

\textbf{Special Case:} While the above includes the possibility that $m=0$, not all of the cases and subcases apply as the ranges may be empty. It is worth noting that when $m=0$ only Case 2, Subcases 2.1 and 2.2 apply.
\end{lemma}

\begin{proof}

For each of the non-trivial subcases above, in order to determine the value of\break $N_{n,m,h} (\rho_2 , \rho_1)$ given a particular range for $\rho_1$ and $\rho_2$, we will first determine the possible \rhopi-partitions of $H_{n+1 , n+1} (\mathbfe{\alpha})$ given the restrictions in (\ref{statement, lemma, main theorem N_m,n,h (rho_1 , rho_2) evaluation, N_n,m,h def}). From there, we can apply Lemmas \ref{lemma, number of reduced matrices with given partition} and \ref{lemma, number of Hankel matrices with given reduced matrix} to determine the number of values that $H_{n+1 , n+1} (\mathbfe{\alpha})$ can take, which is equivalent to $N_{n,m,h} (\rho_2 , \rho_1)$. That is, if we let 
\begin{align*}
&\mathscr{P}_{n,m,h} (\rho_2 , \rho_1 ) \\
&:= \bigg\{ \partition \Big( H_{n+1 , n+1} (\mathbfe{\alpha}) \Big) : \mathbfe{\alpha} \in \mathscr{L}_{2n}^h
	\text{ and } \substack{ 
	\rho_s \Big( H_{m+1 , m+1} (\mathbfe{\alpha}) \Big) = \rho_2 , 
		\hspace{1em} \pi_s \Big( H_{m+1 , m+1} (\mathbfe{\alpha}) \Big) \in \{ 0 , 1 \} \\
	\rho_s \Big( H_{n+1 , n+1} (\mathbfe{\alpha}) \Big) = \rho_1 , 
		\hspace{1em} \pi_s \Big( H_{n+1 , n+1} (\mathbfe{\alpha}) \Big) \in \{ 0 , 1 \} } \bigg\} ,
\end{align*}
then
\begin{align} \label{statement, rho_1 rho_2 possibilities lemma, N_n,m,h (rho_1 , rho_2) formula}
N_{n,m,h} (\rho_2 , \rho_1)
= \sum_{P \in \mathscr{P}_{n,m,h} (\rho_2 , \rho_1 )}
	\sum_{M \in S_{\reduced} (P)}
	\sum_{H \in S_{\hank} (M)} 1 . 
\end{align}
We will use this formula several times. \\

\textbf{Case 1:} $m+1 \leq n \leq 2m$. \\

\underline{Subcase 1.1:} $m \leq h \leq 2n$. \\

Suppose that $h \geq n$. Then, we have that $H_{n+1 , n+1} (\mathbfe{\alpha})$ is a lower skew-triangular Hankel matrix. That is, 
\begin{align*}
H_{n+1 , n+1} (\mathbfe{\alpha})
= \begin{pNiceMatrix}
0 & \Cdots & \Cdots & \Cdots & \Cdots & \Cdots & 0 \\
\Vdots &  &  &  &  &  & \Vdots \\
0 & \Cdots & \Cdots & \Cdots & \Cdots & \Cdots & 0 \\
0 & \Cdots & \Cdots & \Cdots & \Cdots & 0 & \alpha_{h} \\
0 & \Cdots & \Cdots & \Cdots & 0 & \alpha_{h} & \alpha_{h+1} \\
\Vdots &  &  & \Iddots & \Iddots & \Iddots & \Vdots \\
0 & \Cdots & 0 & \alpha_{h} & \alpha_{h+1} & \Cdots & \alpha_{2n} 
\end{pNiceMatrix} ,
\end{align*}
where $\alpha_{h} , \ldots , \alpha_{2n} \in \mathbb{F}_q$. Since we require that $\pi_s \Big( H_{n+1 , n+1} (\mathbfe{\alpha}) \Big) \in \{ 0,1 \}$, we can see that we must have $\alpha_{h} , \ldots , \alpha_{2n-1} = 0$, while $\alpha_{2n}$ is still free to take any value in $\mathbb{F}_q$. \\

Thus, we must have $\rho_2 = 0$ and $\rho_1 = 0$, and $N_{n,m,h} (\rho_2 , \rho_1) = q$. \\

Suppose now that $m \leq h \leq n-1$. Then, we have that $H_{m+1 , m+1} (\mathbfe{\alpha})$ is a lower skew-triangular Hankel matrix. That is, 
\begin{align*}
H_{m+1 , m+1} (\mathbfe{\alpha})
= \begin{pNiceMatrix}
0 & \Cdots & \Cdots & \Cdots & \Cdots & \Cdots & 0 \\
\Vdots &  &  &  &  &  & \Vdots \\
0 & \Cdots & \Cdots & \Cdots & \Cdots & \Cdots & 0 \\
0 & \Cdots & \Cdots & \Cdots & \Cdots & 0 & \alpha_{h} \\
0 & \Cdots & \Cdots & \Cdots & 0 & \alpha_{h} & \alpha_{h+1} \\
\Vdots &  &  & \Iddots & \Iddots & \Iddots & \Vdots \\
0 & \Cdots & 0 & \alpha_{h} & \alpha_{h+1} & \Cdots & \alpha_{2m} 
\end{pNiceMatrix} ,
\end{align*}
where $\alpha_{h} , \ldots , \alpha_{2m} \in \mathbb{F}_q$. Since we require that $\pi_s \Big( H_{m+1 , m+1} (\mathbfe{\alpha}) \Big) \in \{ 0,1 \}$, we can see that we must have $\alpha_{h} , \ldots , \alpha_{2m-1} = 0$. Since in this case we have $n \leq 2m$, we can see that our result for $h \geq n$ (above) applies. \\

So, again, we must have $\rho_2 = 0$ and $\rho_1 = 0$, and $N_{n,m,h} (\rho_2 , \rho_1) = q$. \\

\underline{Subcase 1.2:} $2m-n \leq h \leq m-1$. \\

\textit{Scenario 1.2.a:} There is the trivial scenario where $\rho_2 = 0$ and $\rho_1 = 0$ as above, and $N_{n,m,h} (\rho_2 , \rho_1) = q$. \\

\textit{Scenario 1.2.b and 1.2.c:} Now suppose that $\rho_2$ is non-zero. Since the first $h$ terms of $\mathbfe{\alpha}$ are zero, we can see that we must have $\rho_2 \geq h+1$. Also, by definition of the strict $\rho$-characteristic, we have $\rho_2 \leq m$. So, we must have $h+1 \leq \rho_2 \leq m$. \\

At this point it is helpful to consider the reduced \rhopi-form of $H_{m+1 , m+1} (\mathbfe{\alpha})$. It is of the form
\begin{align} \label{statement, rho_1 rho_2 possibilities lemma, subcase 1.2, red h_(m+1,m+1)}
\reduced \Big( H_{m+1 , m+1} (\mathbfe{\alpha}) \Big)
= \begin{pmatrix}
\begin{matrix}
H_{t_1 +t_2} &  &  &  \\
 & \ddots &  &  \\
 &  & H_{t_1 +2} &  \\
 &  &  & H_{t_1 +1} 
\end{matrix}
& \vline & \mathbf{0} \\
\hline
\mathbf{0} & \vline &
\begin{NiceMatrix}
0 & \Cdots & \Cdots & 0 \\
\Vdots &  &  & \Vdots \\
0 & \Cdots & \Cdots & 0 \\
0 & \Cdots & 0 & \lambda_2 
\end{NiceMatrix}
\end{pmatrix}.
\end{align}
Here, $\lambda_2 \in \mathbb{F}_q$, and the bottom-left quadrant is indicative of the fact that $\pi_s \Big( H_{m+1 , m+1} (\mathbfe{\alpha}) \Big) \in \{ 0,1 \}$. The matrices $H_{t_1 +1} , \ldots , H_{t_1 +t_2}$ are all non-strict lower skew-triangular Hankel matrices. Our indexing begins at $t_1 +1$ because the matrix $\reduced \Big( H_{n+1 , n+1} (\mathbfe{\alpha}) \Big)$, which is an extension of $\reduced \Big( H_{m+1 , m+1} (\mathbfe{\alpha}) \Big)$, will use the blocks $H_{2} , \ldots , H_{t_1}$. Since we are considering the scenario where $\rho_2$ is non-zero, it is not difficult to see that we must have $t_1 , t_2 \geq 1$. \\

Let us denote the \rhopi-partition of $\reduced \Big( H_{m+1 , m+1} (\mathbfe{\alpha}) \Big)$ by $(p' , p'' , p_{t_1 +1} , \ldots , p_{t_1 +t_2})$. The values of $p' , p''$ relate to the bottom-left quadrant of $\reduced \Big( H_{m+1 , m+1} (\mathbfe{\alpha}) \Big)$ above, which we understand fully, and so we are only really interested in $p_{t_1 +1} , \ldots , p_{t_1 +t_2}$. Now, $p_{t_1 +t_2}$ is the number of rows/columns of $H_{t_1 +t_2}$. Since $H_{t_1 +t_2}$ is a non-strict lower skew-triangular Hankel matrix and the first $h$ terms of $\mathbfe{\alpha}$ are zero, we can see that we must have $p_{t_1 +t_2} \geq h+1$. Let us now write
\begin{align} \label{statement, rho_1 rho_2 possibilities lemma, subcase 1.2, s_i def}
s_i 
:= \sum_{j=i}^{t_1 + t_2} p_{j}
\end{align}
for $i = t_1 +1 , \ldots , t_1 + t_2$, and we note that $s_{t_1 +t_2} = p_{t_1 +t_2} \geq h+1$ and $s_{t_1 +1} = \sum_{j=t_1 +1}^{t_1 + t_2} p_{j} = \rho_2$. We can see that the number of possible values of $p_{t_1 +1} , \ldots , p_{t_1 +t_2}$ is equal to the number of ways we can satisfy the conditions
\begin{align*}
h+1 \leq s_{t_1 +t_2} < s_{t_1 +t_2 -1} < \ldots < s_{t_1 +2} < s_{t_1 +1} = \rho_2 ,
\end{align*}
which is simply
\begin{align*}
\binom{\rho_2 - h - 1}{t_2 -1} .
\end{align*}
Of course, we want to know what values $t_2$ can take. The maximum value is the number of integers in $[h+1 , \rho_2]$, which is $\rho_2 - h$. Thus, $1 \leq t_2 \leq \rho_2 -h$. \\

\textit{Scenario 1.2.b:} Now let us consider the reduced \rhopi-form of $H_{n+1 , n+1} (\mathbfe{\alpha})$. The trivial scenario is where
\begin{align*}
\reduced H_{n+1 , n+1} (\mathbfe{\alpha})
= \begin{pmatrix}
\begin{matrix}
H_{t_1 +t_2} &  &  &  \\
 & \ddots &  &  \\
 &  & H_{t_1 +2} &  \\
 &  &  & H_{t_1 +1} 
\end{matrix}
& \vline & \mathbf{0} \\
\hline
\mathbf{0} & \vline &
\begin{NiceMatrix}
0 & \Cdots & \Cdots & 0 \\
\Vdots &  &  & \Vdots \\
0 & \Cdots & \Cdots & 0 \\
0 & \Cdots & 0 & \lambda_1 
\end{NiceMatrix}
\end{pmatrix}.
\end{align*}
That is, $\reduced H_{n+1 , n+1} (\mathbfe{\alpha})$ has more-or-less the same blocks that $\reduced H_{m+1 , m+1} (\mathbfe{\alpha})$ does. Of course, in this scenario, we have $\lambda_2 = 0$, but $\lambda_1$ is still free to take any value in $\mathbb{F}_q$. \\

So, in this scenario, we have $h+1 \leq \rho_2 = \rho_1 \leq m$. Consider the possible values of $t_2$, and the possible values of $p_{t_1 +1} , \ldots , p_{t_1 +t_2}$, that we established above, and apply this to (\ref{statement, rho_1 rho_2 possibilities lemma, N_n,m,h (rho_1 , rho_2) formula}). Using Lemmas \ref{lemma, number of reduced matrices with given partition} and \ref{lemma, number of Hankel matrices with given reduced matrix}, we obtain
\begin{align}
\begin{split} \label{statement, rho_1 rho_2 possibilities lemma, subcase 1.2, rho_1 = rho_2, number of alpha}
N_{n,m,h} (\rho_2 , \rho_1)
= &\sum_{t_2 = 1}^{\rho_2 -h}
	\binom{\rho_2 - h - 1}{t_2 -1}
	\Big( (q-1)^{t_2} q^{\rho_2 - t_2 +1} \Big)
	q^{t_2} \\
= &(q-1) q^{\rho_2 +1} \sum_{t_2 = 0}^{\rho_2 -h-1}
	\binom{\rho_2 - h - 1}{t_2 } (q-1)^{t_2} \\
= &(q-1) q^{\rho_2 +1} q^{\rho_2 - h -1} 
= (q-1) q^{2 \rho_2 -h} . \\
\end{split}
\end{align}

\textit{Scenario 1.2.c:} Again, let us consider the reduced \rhopi-form of $H_{n+1 , n+1} (\mathbfe{\alpha})$. The non-trivial scenario is where
\begin{align*}
\reduced H_{n+1 , n+1} (\mathbfe{\alpha})
= \begin{pmatrix}
\begin{matrix}
H_{t_1 +t_2} &  &  &  \\
 & \ddots &  &  \\
 &  & H_{t_1 +2} &  \\
 &  &  & H_{t_1 +1} 
\end{matrix}
& \vline & \mathbf{0} & \vline & \mathbf{0} \\
\hline
\mathbf{0} & \vline &
\begin{matrix}
H_{t_1} &  &  &  \\
 & \ddots &  &  \\
 &  & H_{3} &  \\
 &  &  & H_{2} 
\end{matrix}
& \vline & \mathbf{0} \\
\hline
\mathbf{0} & \vline & \mathbf{0} & \vline &
\begin{NiceMatrix}
0 & \Cdots & \Cdots & 0 \\
\Vdots &  &  & \Vdots \\
0 & \Cdots & \Cdots & 0 \\
0 & \Cdots & 0 & \lambda_1 
\end{NiceMatrix}
\end{pmatrix} ,
\end{align*}
where $t_1 \geq 2$ and $t_2 \geq 1$, and $H_2 , \ldots , H_{t_1}$ are non-strict lower skew-triangular Hankel matrices. Let us denote the \rhopi-partition of $\reduced \Big( H_{n+1 , n+1} (\mathbfe{\alpha}) \Big)$ by $(p_{1} ' , p_{1} '' , p_{2} , \ldots , p_{t_1 +t_2} )$. Now, note that the bottom-left $(m+1 - \rho_2 ) \times (m+1 - \rho_2 )$ submatrix of (\ref{statement, rho_1 rho_2 possibilities lemma, subcase 1.2, red h_(m+1,m+1)}), namely
\begin{align*}
\begin{pNiceMatrix}
0 & \Cdots & \Cdots & 0 \\
\Vdots &  &  & \Vdots \\
0 & \Cdots & \Cdots & 0 \\
0 & \Cdots & 0 & \lambda_2 
\end{pNiceMatrix} ,
\end{align*}
is a top-left submatrix of $H_{t_1}$. Considering this, and the fact that $H_{t_1}$ is a non-strict lower skew-triangular Hankel matrix, we can see that the smallest value that $p_{t_1}$ (the number of rows/columns of $H_{t_1}$) can take is $2m+1 - 2 \rho_2$ which occurs when $\lambda_2$ is non-zero. Note that this requires that
\begin{align*}
n 
\geq p_{t_1} + (p_{t_1 + 1} + \ldots + p_{t_1 +t_2})
\geq ( 2m+1 - 2 \rho_2 ) + (\rho_2 )
= 2m+1 - \rho_2 .
\end{align*}
In particular, we must have 
\begin{align} \label{statement, rho_1 rho_2 possibilities lemma, subcase 1.2, rho_1 > rho_2, rho_2 geq 2m+1 - n}
\rho_2 \geq 2m+1 - n .
\end{align}
However, this is already satisfied by the requirement that $\rho_2 \geq h+1$ which we established earlier, because in this subcase we have $h+1 \geq 2m+1-n$. \\

Now, the smallest value that $\rho_1$ can take occurs when $t_1 = 2$ and it is $p_{t_1} + (p_{t_1 + 1} + \ldots + p_{t_1 +t_2})$; in turn, as we established above, the minimum value this can take is $( 2m+1 - 2 \rho_2 ) + (\rho_2 ) = 2m+1 - \rho_2$. Obviously the largest value $\rho_1$ can take is $n$. Thus, we have $h+1 \leq \rho_2 \leq m$ and $2m+1 - \rho_2 \leq \rho_1 \leq n$. \\

We now wish to determine the range of $t_1$, and the number of values that $p_{2} , \ldots , p_{t_1}$ can take. To this end, let us extend the definition of $s_i$ in (\ref{statement, rho_1 rho_2 possibilities lemma, subcase 1.2, s_i def}) to include $i=2 , \ldots , t_1$. We have already established that
\begin{align*}
s_{t_1} 
= p_{t_1} + (p_{t_1 + 1} + \ldots + p_{t_1 +t_2})
= (2m+1 - 2 \rho_2 ) + (\rho_2 )
= 2m+1 - \rho_2
\end{align*}
and
\begin{align*}
s_2
= p_{2} + \ldots + p_{t_1 + t_2}
= \rho_1 .
\end{align*}
Thus, the number of possible values that $p_{t_1} , \ldots , p_{2}$ can take is equal to the number of ways we can satisfy the conditions
\begin{align*}
2m+1 - \rho_2 \leq s_{t_1} < s_{t_1 -1} < \ldots < s_{3} < s_{2} = \rho_1 ,
\end{align*}
which is simply
\begin{align*}
\binom{\rho_1 + \rho_2 - 2m - 1}{t_1 -2} .
\end{align*}
Also, the maximum value that $t_1 -1$ can take is the number of integers in $[2m+1 - \rho_2 , \rho_1]$, which is $\rho_1 + \rho_2 - 2m$. Thus, $2 \leq t_1 \leq \rho_1 + \rho_2 +1 - 2m$. \\

We now wish to determine the value of $N_{n,m,h} (\rho_2 , \rho_1)$. Using (\ref{statement, rho_1 rho_2 possibilities lemma, N_n,m,h (rho_1 , rho_2) formula}) and Lemmas \ref{lemma, number of reduced matrices with given partition} and \ref{lemma, number of Hankel matrices with given reduced matrix}, we obtain
\begin{align*}
&N_{n,m,h} (\rho_2 , \rho_1) \\
= &\sum_{t_2 = 1}^{\rho_2 -h} \binom{\rho_2 - h - 1}{t_2 -1}
	\sum_{t_1 = 2}^{\rho_1 + \rho_2 +1 - 2m} \binom{\rho_1 + \rho_2 - 2m - 1}{t_1 -2}
	\bigg[ (q-1)^{t_1 + t_2 -1 } q^{\rho_1 - t_1 - t_2 +2} \bigg] q^{t_1 + t_2 -1} \\
= & (q-1)^2 q^{\rho_1 +1} \bigg( \sum_{t_2 = 0}^{\rho_2 -h -1} \binom{\rho_2 - h - 1}{t_2 } (q-1)^{t_2}
	\times \sum_{t_1 = 0}^{\rho_1 + \rho_2 - 2m -1} \binom{\rho_1 + \rho_2 - 2m - 1}{t_1 } (q-1)^{t_1} \bigg) \\
= &(q-1)^2 q^{2 \rho_1 + 2 \rho_2 -2m -h -1} . \\
\end{align*}

\underline{Subcase 1.3:} $0 \leq h \leq 2m-n-1$. \\

This subcase does not apply when $n=2m$, because the range of $h$ will be empty. \\

Now, scenarios a and b are the same as in Subcase 1.2. It is only Scenario c that differs slightly. The situation is the same until we reach (\ref{statement, rho_1 rho_2 possibilities lemma, subcase 1.2, rho_1 > rho_2, rho_2 geq 2m+1 - n}) which states that $\rho_2 \geq 2m+1-n$. In Subcase 1.2 we had $h+1 \geq 2m+1-n$, which meant that (\ref{statement, rho_1 rho_2 possibilities lemma, subcase 1.2, rho_1 > rho_2, rho_2 geq 2m+1 - n}) was implied by the condition $\rho_2 \geq h+1$ which was established earlier. However, in this subcase, we have $h+1 \leq 2m-n$, and so the two conditions $\rho_2 \geq 2m+1-n$ and $\rho_2 \geq h+1$ are equivalent to the single condition $\rho_2 \geq 2m+1-n$. \\

So, in this subcase and scenario, we have $2m+1-n \leq \rho_2 \leq m$ and $2m+1 - \rho_2 \leq \rho_1 \leq n$. The value of $N_{n,m,h} (\rho_2 , \rho_1)$ is determined in exactly the same way, and so we have
\begin{align*}
N_{n,m,h} (\rho_2 , \rho_1) 
= (q-1)^2 q^{2 \rho_1 + 2 \rho_2 -2m -h -1} . \\
\end{align*}

\textbf{Case 2:} $n \geq 2m+1$. \\

\underline{Subcase 2.1:} $n \leq h \leq 2n$. \\

This is the same as the first part of Subcase 1.1. That is, we must have $\rho_2 , \rho_1 = 0$, and $N_{n,m,h} (\rho_2 , \rho_1)  = q$. \\

\underline{Subcase 2.2:} $2m \leq h \leq n-1$. \\

Since $h \geq 2m$, we must have that $\rho_2 = 0$ and
\begin{align*}
\reduced \Big( H_{m+1 , m+1} (\mathbfe{\alpha}) \Big)
= \begin{pNiceMatrix}
0 & \Cdots & \Cdots & 0 \\
\Vdots &  &  & \Vdots \\
0 & \Cdots & \Cdots & 0 \\
0 & \Cdots & 0 & \lambda_2 
\end{pNiceMatrix} ,
\end{align*}
where $\lambda_2 = 0$, unless $h=2m$ in which case $\lambda_2 \in \mathbb{F}_q$. \\

The trivial scenario is when $\rho_2 , \rho_1 = 0$ and $N_{n,m,h} (\rho_2 , \rho_1)  = q$, as in Subcase 2.1. \\

The non-trivial scenario is when $\rho_1$ is non-zero. Let us write
\begin{align*}
\reduced H_{n+1 , n+1} (\mathbfe{\alpha})
= \begin{pmatrix}
\begin{matrix}
H_{t} &  &  &  \\
 & \ddots &  &  \\
 &  & H_{3} &  \\
 &  &  & H_{2} 
\end{matrix}
& \vline & \mathbf{0} \\
\hline
\mathbf{0} & \vline &
\begin{NiceMatrix}
0 & \Cdots & \Cdots & 0 \\
\Vdots &  &  & \Vdots \\
0 & \Cdots & \Cdots & 0 \\
0 & \Cdots & 0 & \lambda_1 
\end{NiceMatrix}
\end{pmatrix} .
\end{align*}
The matrices $H_t , \ldots , H_2$ are non-strict lower skew-triangular Hankel matrices, and since $\rho_1 \neq 0$ we must have that $t \geq 2$. Also, $\lambda_1 \in \mathbb{F}_q$. Let $(p_1 ' , p_1 '' , p_2 , \ldots , p_t )$ be the \rhopi-partition of $H_{n+1 , n+1} (\mathbfe{\alpha})$. \\

Now, since $H_t$ is a non-strict lower skew-triangular Hankel matrix, and since the first $h$ terms of $\mathbfe{\alpha}$ are zero, we can see that the smallest value $p_t$ can take is $h+1$ and this occurs when $\alpha_h \neq 0$. \\

The smallest value $\rho_1$ can take is also $h+1$, which occurs when $t=2$ and $p_t = h+1$. The maximum value $\rho_1$ can take is $n$. Thus, we have $h+1 \leq \rho_1 \leq n$. \\

As previously, we define
\begin{align*}
s_i 
:= \sum_{j=i}^{t} p_{j}
\end{align*}
for $i=t , \ldots , 2$. We have already established that
\begin{align*}
s_t = p_t \geq h+1
\end{align*}
and 
\begin{align*}
s_2 = p_2 + \ldots p_t = \rho_1 .
\end{align*}
Thus, the number of possible values that $p_2 , \ldots , p_t$ can take is equal to the number of ways that we can satisfy the conditions
\begin{align*}
h+1 \leq s_t < s_{t-1} < \ldots < s_3 < s_2 = \rho_1 ,
\end{align*}
which is simply
\begin{align*}
\binom{\rho_1 - h -1}{t-2} .
\end{align*}
Also, the range of $t$ is $2 \leq t \leq \rho_1 - h +1$. Thus, once again using (\ref{statement, rho_1 rho_2 possibilities lemma, N_n,m,h (rho_1 , rho_2) formula}) and Lemmas \ref{lemma, number of reduced matrices with given partition} and \ref{lemma, number of Hankel matrices with given reduced matrix}, we obtain
\begin{align*}
&N_{n,m,h} (\rho_2 , \rho_1) \\
= &\sum_{t = 2}^{\rho_1 -h+1} \binom{\rho_1 - h -1}{t-2}
	\Big( (q-1)^{t -1 } q^{\rho_1 - t +2} \Big) q^{t -1} \\
= &(q-1) q^{\rho_1 +1} \sum_{t = 0}^{\rho_1 -h-1} \binom{\rho_1 - h -1}{t} (q-1)^{t } \\
= &(q-1) q^{2\rho_1 -h} . \\
\end{align*}

\underline{Subcase 2.3:} $m \leq h \leq 2m-1$. \\

This is very similar to Subcase 2.2. As in the second part of Subcase 1.1, we must have that $\rho_2 = 0$ and
\begin{align} \label{statement, rho_1 rho_2 possibilities lemma, subcase 2.2, red h_(m+1,m+1) form}
\reduced \Big( H_{m+1 , m+1} (\mathbfe{\alpha}) \Big)
= \begin{pNiceMatrix}
0 & \Cdots & \Cdots & 0 \\
\Vdots &  &  & \Vdots \\
0 & \Cdots & \Cdots & 0 \\
0 & \Cdots & 0 & \lambda_2 
\end{pNiceMatrix} ,
\end{align}
where $\lambda_2 \in \mathbb{F}_q$. There is the trivial scenario where $\rho_1 = 0$, for which we have $N_{n,m,h} (\rho_2 , \rho_1) = q$. \\

There is also the non-trivial scenario where $\rho_1$ is non-zero. Again, we write
\begin{align*}
\reduced H_{n+1 , n+1} (\mathbfe{\alpha})
= \begin{pmatrix}
\begin{matrix}
H_{t} &  &  &  \\
 & \ddots &  &  \\
 &  & H_{3} &  \\
 &  &  & H_{2} 
\end{matrix}
& \vline & \mathbf{0} \\
\hline
\mathbf{0} & \vline &
\begin{NiceMatrix}
0 & \Cdots & \Cdots & 0 \\
\Vdots &  &  & \Vdots \\
0 & \Cdots & \Cdots & 0 \\
0 & \Cdots & 0 & \lambda_1 
\end{NiceMatrix}
\end{pmatrix} 
\end{align*}
and let $(p_1 ' , p_1 '' , p_2 , \ldots , p_t )$ be the \rhopi-partition. The difference between this subcase and Subcase 2.2 is that here we have that the first $2m$ terms of $\mathbfe{\alpha}$ are zero (due to (\ref{statement, rho_1 rho_2 possibilities lemma, subcase 2.2, red h_(m+1,m+1) form})), whereas in Subcase 2.2 
it was the first $h$ terms of $\mathbfe{\alpha}$ that were zero. \\

Nonetheless, we proceed in a similar manner and we see that the smallest value that $p_t$ can take is $2m+1$, and the range of $\rho_1$ is $2m+1 \leq \rho_1 \leq n$. The number of possible values that $p_2 , \ldots , p_t$ can take is equal to the number of ways that we can satisfy the conditions
\begin{align*}
2m+1 \leq s_t < s_{t-1} < \ldots < s_3 < s_2 = \rho_1 ,
\end{align*}
which is simply
\begin{align*}
\binom{\rho_1 - 2m -1}{t-2} .
\end{align*}
Also, the range of $t$ is $2 \leq t \leq \rho_1 - 2m +1$. Thus, 
\begin{align*}
&N_{n,m,h} (\rho_2 , \rho_1) \\
= &\sum_{t = 2}^{\rho_1 -2m+1} \binom{\rho_1 - 2m -1}{t-2}
	\Big( (q-1)^{t -1 } q^{\rho_1 - t +2} \Big) q^{t -1} \\
= &(q-1) q^{2\rho_1 -2m} . \\
\end{align*}

\underline{Subcase 2.4:} $0 \leq h \leq m-1$. \\

\textit{Scenario 2.3.a:} Again, there is the trivial scenario where $\rho_2 , \rho_1 = 0$, for which we have $N_{n,m,h} (\rho_2 , \rho_1)  =q$. \\

\textit{Scenario 2.3.b:} There is there scenario where $\rho_2 = 0$ but $\rho_1$ is non-zero. As in Subcase 2.3, the possible values of $\rho_1$ is $2m+1 \leq \rho_1 \leq n$, and $N_{n,m,h} (\rho_2 , \rho_1) =(q-1) q^{2\rho_1 -2m}$. \\

\textit{Scenario 2.3.c and 2.3.d:} It is possible that $\rho_2$ is non-zero. As in Subcases 1.2 and 1.3, the values that $\rho_2$ can take are $h+1 \leq \rho_2 \leq m$. \\

\textit{Scenario 2.3.c:} In this scenario, we have $\rho_1 = \rho_2$. This is virtually identical to Scenario 1.2.b, and by the same reasoning we have $N_{n,m,h} (\rho_2 , \rho_1) = (q-1) q^{2 \rho_2 - h}$. \\

\textit{Scenario 2.3.d:} In this scenario, we have $\rho_1 > \rho_2$. We proceed as in Scenario 1.2.c, until (\ref{statement, rho_1 rho_2 possibilities lemma, subcase 1.2, rho_1 > rho_2, rho_2 geq 2m+1 - n}) which states that $\rho_2 \geq 2m+1-n$. Since in this case we have $n \geq 2m+1$, we can see that (\ref{statement, rho_1 rho_2 possibilities lemma, subcase 1.2, rho_1 > rho_2, rho_2 geq 2m+1 - n}) is already implied by the condition $h+1 \leq \rho_2 \leq m$. The possible values of $\rho_1$ are still $2m+1 - \rho_2 \leq \rho_1 \leq n$, and again we have $N_{n,m,h} (\rho_2 , \rho_1) = (q-1)^2 q^{2 \rho_1 + 2 \rho_2 -2m-h-1}$.
\end{proof}

\vspace{1em}

\textbf{Acknowledgements:} This research was conducted during a postdoctoral fellowship funded by the Leverhulme Trust research project grant ``Moments of $L$-functions in Function Fields and Random Matrix Theory'' (grant number RPG-2017-320) secured by Julio Andrade. The author is most grateful for this support.

%\addcontentsline{toc}{section}{References}
\bibliography{YiasemidesBibliography1}{}
\bibliographystyle{YiaseBstNumer1}

\end{document}